\documentclass[12pt]{amsart}

\usepackage{amsthm}
\usepackage{amsmath}
\usepackage{amsfonts,amssymb,latexsym}
\usepackage[english]{babel}
\usepackage{epsfig,epsf}
\usepackage{tikz}
\usepackage{enumerate}
\usepackage{color}
\newtheorem{thm}{Theorem}
\newtheorem{lem}{Lemma}

\newtheorem{prop}{Proposition}

\newtheorem{rem}{Remark}

\newcommand{\N}{{\mathbb N}}
\newcommand{\Z}{{\mathbb Z}}

\newcommand{\R}{{\mathbb R}}

\newcommand{\Tor}{{\mathbb T}}

\title[Small sumsets in $\R$]{Small sumsets in $\R$ : a continuous $3k-4$ theorem}
\author[Anne de Roton]{Anne de Roton}
\thanks{Soutien de l'ANR C\ae sar, ANR 12 - BS01 - 0011}
\email{anne.de-roton@univ-lorraine.fr}
\address{Universit\' e de Lorraine, Institut Elie Cartan de Lorraine, UMR 7502, Vandoeuvre-l\`es-Nancy, F-54506, France\\
CNRS, Institut Elie Cartan de Lorraine, UMR 7502, Vandoeuvre-l\`es-Nancy, F-54506, France.}

\begin{document}

\begin{abstract}
We prove a continuous Freiman's $3k-4$ theorem for small sumsets in $\R$ by using some ideas from Ruzsa's work on measure of sumsets in $\R$ as well as some graphic representation of density functions of sets. We thereby get some structural properties of $A$, $B$ and $A+B$ when $\lambda(A+B)<\lambda(A)+\lambda(B)+\min(\lambda(A),\lambda(B))$. We also give some structural information for sets of large density with small sumset and characterize the extremal sets for which equality holds in the lower bounds for $\lambda(A+B)$.

\end{abstract}

\maketitle

\section{Introduction}
Inverse problems for small sumsets study the structural properties of sets $A$ and $B$ when their sumset $A+B=\{a+b,\, a\in A, \, b\in B\}$ is small (see \cite{Tao_Vu} or \cite{Nathanson} for an overview on this subject).
In 1959, Freiman \cite{Freiman1959} proved that a set $A$ of integers such that $|A+A|\leq 3|A|-4$, where $|A|$ denotes the number of elements in $A$, is contained in an arithmetic progression of length $|A+A|-|A|+1$. This result is usually referred to as Freiman's $(3k-4)$ theorem. It has been refined in many ways and generalised to finite sets in other groups or semi-groups. The most complete version of this theorem for integers can be found in \cite{Gr}, chapter 7.

In this paper, we consider the addition of two bounded sets $A$ and $B$ of real numbers. We establish a continuous analogue of the complete Freiman's $(3k-4)$ theorem and study the structures of the extremal sets for which the lower bounds are attained. We also prove some results on sets of real numbers so far unknown for sets of integers. Our first main result can be read as follows ($\lambda$ is the inner Lebesgue measure on $\R$ and $\rm{diam}(A)=\sup(A)-\inf(A)$ is the diameter of $A$).
\begin{thm}\label{intro_3k-4}
Let $A$ and $B$ be measurable bounded subsets of $\mathbb{R}$ such that $\lambda(A),\lambda(B)\not=0$. 
If 
\begin{enumerate}
\item[i)] either $\lambda(A+B)< \lambda(A)+\lambda(B)+\min(\lambda(A),\lambda(B))$; 
\item[ii)] or ${\rm{diam}}(B)\leq {\rm{diam}}(A)$ and $\lambda(A+B)< \lambda(A)+2\lambda(B)$;
\end{enumerate}
then 
\begin{enumerate}
\item ${\rm{diam}}(A)\leq \lambda(A+B)-\lambda(B)$, 
\item ${\rm{diam}}(B)\leq \lambda(A+B)-\lambda(A)$, 
\item there exists an interval $I$ of length at least $\lambda(A)+\lambda(B)$ included in $A+B$.
\end{enumerate}
\end{thm}

Beyond the result themselves, what is striking is that the proof in the continuous setting is much easier to understand than in the discrete setting. The first two statements under hypothesis (i) are a straightforward application of Ruzsa's results in \cite{Ruz}. This nice paper of Ruzsa seems to have been overlooked whereas his ideas may lead to further results in the continuous setting that may even yield some improvements in the discrete one. This part of the theorem has already partially been proved by M. Christ in  \cite{Cr} (for $A=B$). \\
In \cite{Ruz}, Ruzsa improved on the well-known lower bound 
$$\lambda(A+B)\geq \lambda(A)+\lambda(B)$$
and proves that, if $\lambda(A)\leq \lambda(B)$, this can be replaced by 
$$\lambda(A+B)\geq \lambda(A)+\min(\rm{diam}(B),\lambda(A)+\lambda(B)).$$
The main idea of his proof is to transfer the sum in $\R$ in a sum in $\Tor=\R/\Z$. 
Ruzsa considers the sets $A$ and $B$ of real numbers as sets of numbers modulo $D_B$, the diameter of $B$. 
In this setting, he can use a rescaled version of Raikov's theorem \cite{Raikov} as well as the fact that if $x\in[0,D_B]$ belongs to $A$ and if $B$ is a closed set then $x, x+D_B\in A+B$.\\
Ruzsa's result directly yields the first two statements of our theorem under condition (i). To get these statements from condition (ii), we need to use Ruzsa's arguments in a slightly different way. This part is the continuous analogue of Freiman's $3k-4$ theorem in \cite{Freiman1959} as generalised to the sum of two distinct sets by Freiman  \cite{Freiman1962}, Lev and Smeliansky \cite{Lev_Smeliansky} and Stanchescu \cite{stanchescu}.\\
As far as we know, the third consequence (the existence of an interval $I$ of length at least $\lambda(A)+\lambda(B)$ included in $A+B$) was not known in the continuous setting. The discrete analogue of our third statement was proved by Freiman in \cite{Freiman2009IAP11} in the special case $A=B$. It has been generalised to the case $A\not=B$ by Bardaji and Grynkiewicz's theorem in \cite{Bardaji_Grynkiewicz}. An exposition of these results can be found in \cite{Gr}, chapter 7. We could adapt their proof but a simpler proof follows from some density arguments in the continuous setting. Actually, we think that the ideas are more natural in the continuous setting where a graphic illustration leads to the result. We hope that this sheds some new light on inverse results for integers too. \\
This third statement is a consequence of the simple remark that if the sum of the densities of $A$ and $B$ on $[0,x]$ is strictly larger than $x$, then $x$ can be written as a sum of an element in $A$ and an element in $B$ and of a symmetric result. This allows us to partition $[0,D_A]$ into three sets : a subset $Z_1$ of $A+B$, a subset $Z_3$ of $A+B-D_A$ and their complementary set $Z_2$ included in both $A+B$ and $A+B-D_A$. To go from $Z_1$ to $Z_3$ and reciprocally, one need to cross $Z_2$. The proof relies on the fact that there is only one such crossing under the hypothesis and on a lower bound for the measure of $Z_2$. 

\bigskip

The graphic interpretation leads to a relaxed inverse Freiman theorem for sets of large density with small sumset. Namely, we prove the following result.
\begin{thm}\label{intro_relaxed_3k-4}
Let $A$ and $B$ be measurable bounded subsets of $\mathbb{R}$ such that $D_B:=\rm{diam}(B)\leq D_A:=\rm{diam}(A)$ and $\Delta:=\lambda(A)+\lambda(B)-D_A>0$. Let $m$ be a non negative integer.
If $$\lambda(A+B)< D_A+\lambda(B)+(m+1)(D_A-D_B+\Delta)$$ then the sum $A+B$ contains a union of at most $m$ disjoint intervals $K_1, K_2, \cdots K_n$ ($n\leq m$), each of length at least $2\Delta+D_A-D_B$ such that the measure of this union of intervals is at least $D_A+(2n+1)\Delta$.
\end{thm}

With this weak hypothesis, a description of the sets $A$ and $B$ can be given. This is netherveless a rather vague description. On the contrary, we get a precise description of sets $A$ and $B$ for which the lower bound for the measure of $A+B$ is attained. 
\begin{thm}\label{intro_equalarge}
Let $A$ and $B$ be some bounded sets of real numbers such that $D_B\leq D_A$ and $\lambda(A+B)=D_B+\lambda(A)< \lambda(A)+2\lambda(B)$. Write $A'=A-\inf(A)$ and $B'=B-\inf(B)$.
Then there exists two positive real numbers $b$ and $c$ such that $b,c\leq D_B$, the interval $I=(b,D_A-c)$ has size at least $\lambda(A)+\lambda(B)-D_B=\Delta+D_A-D_B$ and
\begin{itemize}
\item $A_1=A'\cap[Õ0, b]$ is $B'$-stable, i.e. $A'\cap [0,b]=(A'+B')\cap[0,b]$;
\item $A_2=D_A-(A'\cap[D_A-c, D_A])$ is $B'$-stable;
\item $I=(b,D_A-c)\subset A'$;
\item $(b,D_A+D_B-c)\subset A'+B'$
 \end{itemize}
 up to a set of measure $0$.
The sets $A'$ and $A'+B'$ may each be partitioned into three parts as follows
$$A'=A_1\cup I\cup (D_A-A_2), \quad 
A'+B'=A_1\cup(b,D_A+D_B-c)\cup(D_A+D_B-A_2)$$
 Furthermore $B'$ is a disjoint union of three sets $B=B_1\cup B_I\cup (D_B-B_2)$ with 
 \begin{itemize}
\item $B_1\subset A_1$, $B_2\subset A_2$, $B_I\subset(b,D_B-c)$;
\item $\lambda(B_1\cap[0,x])\leq x^2/c_1$ for $x\leq c_1=\sup\{x\, :\, \lambda(A'\cap[0,x])\leq \frac12x\}$ 
\item $\lambda(B_2\cap[0,x])\leq x^2/c_2$ for $x\leq c_2=D_A-\sup\{x\, :\, \lambda(A'\cap(D_A-x,D_A))\leq \frac12x\}$
 \end{itemize}
 up to a set of measure $0$.
\end{thm}
We also prove a similar result under the hypothesis $\lambda(A+B)=\lambda(B)+D_A$. Only the result on the density of the sets near the border of the interval requires some real new work. The other results are a consequence of the previous observations on graph functions. To study the densities, one uses Ruzsa's lower bound for the sum $A+B$ in terms of the ratio $\lambda(A)/\lambda(B)$. Precisely, Ruzsa proved the following theorem \cite{Ruz}
 \begin{thm}[Ruzsa]\label{intro_cor_ruzsa}
Let $A$ and $B$ be bounded subsets of $\mathbb{R}$ such that $\lambda(B)\not=0$. Write $D_B={\rm diam}(B)$ and define $K\in\N^*$ and $\delta\in \R$ such that
\begin{equation}\label{defKdelta}
\frac{\lambda(A)}{\lambda(B)} =\frac{K(K-1)}{2}+K\delta, \quad 0\leq\delta<1.
\end{equation}
Then we have
\begin{equation*}
 \lambda(A+B)\geq \lambda(A)+\min({\rm diam}(B),(K+\delta)\lambda(B)).
 \end{equation*}
 \end{thm}
A simple remark yields an improvement of this lower bound when $\rm{diam}(A)/\rm{diam}(B)\leq K$ and a partial result on sets $B$ such that $\lambda(A+B)<\lambda(A)+(K+\delta)\lambda(B)$.
The extremal sets in this context can also be described, this time in a very precise way. 
\begin{thm}\label{small_equa}
Let $A$ and $B$ be measurable bounded subsets of $\mathbb{R}$ such that $\lambda(A),\lambda(B)\not=0$. 
Let $K\in\N$ and $\delta\in[0,1)$ be such that 
$$\frac{\lambda(A)}{\lambda(B)} =\frac{K(K-1)}{2}+K\delta \quad\mbox{ and }\quad \lambda(A+B)=\lambda(A)+(K+\delta)\lambda(B)<\lambda(A)+ D_B$$
where $D_B={\rm diam}(B)$.
Then $A$ and $B$ are translates of sets $A'$ and $B'$ of the form
$$B'=[0,b_1]\cup[D_B-b_2, D_B]; \quad A'=\bigcup_{k=1}^K \left[(k-1)(D_B-b_2),(k-1)D_B+(K-k)b_1+\delta b\right]$$
with $b_1,b_2\geq 0$ and $b_1+b_2=b=\lambda(B)$.
\end{thm}

Section 2 of this paper is devoted to the continuous $3k-4$ theorem. In section 3, we prove the relaxed inverse Freiman's theorem. We describe in section 4 the large sets for which the lower bound in Ruzsa's inequality is attained. Finally, the last section is devoted to the characterisation of the small sets for which the lower bound in Ruzsa's inequality is attained.

\section{Continuous Freiman $3k-4$ theorem}
In this paper, $\lambda$ will denote the inner Lebesgue measure on $\R$. Given a bounded set $S$ of real numbers, we define its diameter $D_S=\rm{diam}(S)=\sup(S)-\inf(S)$.  

\bigskip

In \cite{Ruz}, Ruzsa obtains some lower bounds for the measure of the sum $A+B$ of two subsets $A$ and $B$ of real numbers in terms of $\lambda(A)$, $\lambda(B)$ and $\rm{diam}(B)$. We state here one of his intermediate results. 

\begin{lem}[Ruzsa  \cite{Ruz}]\label{ruzsa1}
Let $A$ and $B$ be bounded subsets of $\mathbb{R}$. Write $D_B={\rm diam}(B)$.
Then we have either
\begin{equation}\label{min_sum1}
 \lambda(A+B)\geq \lambda(A)+{\rm diam}(B) 
 \end{equation}
 or 
 \begin{equation}\label{min_sum1}
 \lambda(A+B)\geq \frac{k+1}{k}\lambda(A)+\frac{k+1}{2}\lambda(B)
 \end{equation}
 with $k$ the positive integer defined by  
 $$k=\max\{k'\in\N \, : \, \exists x\in[0,D_B) \, : \, \#\{n\in\N \, : \, x+nD_B\in A\}\geq k'\}.$$
\end{lem}
As a corollary, Ruzsa derives Theorem \ref{intro_cor_ruzsa}. In the following theorem, we improve this result for small sets $A$ and $B$ such that $D_A/D_B$ is small. This gives a partial answer to one of the questions asked by Ruzsa in \cite{Ruz}. Namely Ruzsa asked for a lower bound depending on the measures and the diameters of the two sets $A$ and $B$.
\begin{thm}\label{cor_ruzsa}
Let $A$ and $B$ be bounded subsets of $\mathbb{R}$ such that $\lambda(B)\not=0$. Write $D_B={\rm diam}(B)$, $D_A={\rm diam}(A)$ and define $K\in\N^*$ and $\delta\in \R$ as in \eqref{defKdelta}.
Then we have either
\begin{equation*}
 \lambda(A+B)\geq \lambda(A)+{\rm diam}(B) 
 \end{equation*}
 or 
 \begin{equation}\label{min_sum2}
 \lambda(A+B)\geq  \lambda(A)+(K+\delta)\lambda(B).
 \end{equation}
 Furthermore, if $D_A/D_B\leq K$, then 
 then \eqref{min_sum2} can be replaced by the better estimate
 $$\lambda(A+B)\geq\frac{\lceil{D_A/D_B}\rceil+1}{\lceil{D_A/D_B}\rceil}\lambda(A)+\frac{\lceil{D_A/D_B}\rceil+1}{2}\lambda(B).$$
 \end{thm}

\begin{rem}
This theorem is mostly due tu Ruzsa in \cite{Ruz}. Our only contribution consists in noticing that the lower bound can be improved in case  
$D_A/D_B\leq K$. In case $D_A\leq D_B$, this remark yields the lower bound 
\begin{equation}\label{minD_A<D_B}
 \lambda(A+B)\geq \lambda(A)+\min({\rm diam}(B), \lambda(A)+\lambda(B)).
\end{equation}
\end{rem}

\begin{proof}
Let us assume that $ \lambda(A+B)< \lambda(A)+{\rm diam}(B)$. Then by
Lemma \ref{ruzsa1}, $\lambda(A+B)\geq f(K_A)$ holds
 with $f(k)=\frac{k+1}{k}\lambda(A)+\frac{k+1}{2}\lambda(B)$ and $$K_A=\max\{k'\in\N \, : \, \exists x\in[0,D_B) \, : \, \#\{n\in\N \, : \, x+nD_B\in A\}\geq k'\}.$$
 As noticed by Ruzsa, 
the sequence $(f(k))_{k\geq 1}$ is decreasing for $k\leq K$ and increasing for $k\geq K$ with $K$ the integer defined by \eqref{defKdelta}.
Therefore $f(k)$ is minimal for $k=K$ and we get the lower bound
\begin{equation*}
\lambda(A+B)\geq f(K_A)\geq f(K)=\lambda(A)+(K+\delta)\lambda(B).
\end{equation*}
On the other hand, it is clear that $K_A\leq \lceil D_A/D_B\rceil$. Therefore, if $D_A/D_B\leq K$ 
then we have
$\lambda(A+B)\geq f(\lceil D_A/D_B\rceil)$.\end{proof}

As an immediate consequence of Theorem \ref{cor_ruzsa} we derive the following proposition.
\begin{prop}\label{K+delta}
Let $A$ and $B$ be bounded subsets of $\mathbb{R}$ such that $\lambda(B)\not=0$. Write $D_B={\rm diam}(B)$, $D_A={\rm diam}(A)$ and define $K\in\N^*$ and $\delta\in \R$ as in \eqref{defKdelta}.
If we have either
$$\lambda(A+B)<  \lambda(A)+(K+\delta)\lambda(B).$$
or $D_A/D_B\leq K$ and 
$$\lambda(A+B)< \frac{\lceil{D_A/D_B}\rceil+1}{\lceil{D_A/D_B}\rceil}\lambda(A)+\frac{\lceil{D_A/D_B}\rceil+1}{2}\lambda(B)$$
then ${\rm diam}(B) \leq  \lambda(A+B)-\lambda(A)$.
 \end{prop}

As noticed by Ruzsa in \cite{Ruz},  in case $0<\lambda(A)\leq\lambda(B)$ \eqref{min_sum2} yields
$\lambda(A+B)\geq \min(2\lambda(A)+\lambda(B),\lambda(A)+\rm{diam}(B))$. In Theorem \ref{intro_3k-4} we extend his result to get a continuous version of the complete Freiman $3k-4$ theorem for sets of integers.

%
As a consequence of our proof, we can derive, as in the discrete case, that the interval $I$ included in $A+B$ has lower bound $e:=\sup\{x\in [0,D_A], x\not\in A+B\}$ and upper bound $c:=\inf\{x\in[D_A,D_A+D_B], x\not\in A+B\}$. Furthermore we get that $\lambda(A\cap[0,x])+\lambda(B\cap[0,x])> x$ for $x>e$ and $\lambda(A\cap[0,x+D_A-D_B])+\lambda(B\cap[0,x])< x+\lambda(A)+\lambda(B)-D_B$ for $x<c-D_A$.

\begin{proof}[Proof of the first two items of Theorem \ref{intro_3k-4}]
We use the notation introduced in Theorem \ref{cor_ruzsa}. 
We first prove that each hypothesis yields the first two points. This is the straightforward part of the proof. The proof of the third item is more demanding and will require some intermediate results.
\begin{itemize}
\item Let us assume that we have hypothesis (i) and that $\lambda(A)\leq \lambda(B)$, say. Then we have $\lambda(A)/\lambda(B)\leq 1$ thus $K=1$ or $(K,\delta)=(2,0)$ and Theorem \ref{cor_ruzsa} gives either $\lambda(A+B)\geq 2\lambda(A)+\lambda(B)$ or  $\lambda(A+B)\geq \lambda(A)+{\rm diam}(B)$. Since  $\lambda(A+B)< 2\lambda(A)+\lambda(B)$, we must have 
${\rm{diam}}(B)\leq \lambda(A+B)-\lambda(A)$. \\
On the other hand $$\frac{\lambda(B)}{\lambda(A)}=\frac{K'(K'-1)}{2}+K'\delta'$$ with $K'\geq 2$ and $0\leq\delta'\leq 1$ thus $\lambda(A+B)< 2\lambda(A)+\lambda(B)\leq\lambda(A)+2\lambda(B)\leq \lambda(A)+(K'+\delta')\lambda(B)$ and Theorem \ref{cor_ruzsa} yields ${\rm{diam}}(A)\leq \lambda(A+B)-\lambda(B)$.

\item If ${\rm{diam}}(B)\leq {\rm{diam}}(A)$, then \eqref{minD_A<D_B} with $(A,B)$ in place of $(B,A)$ gives either $\lambda(A+B)\geq \lambda(A)+2\lambda(B)$ or  $\lambda(A+B)\geq \lambda(B)+{\rm diam}(A)$ thus hypothesis (ii) gives ${\rm{diam}}(A)\leq \lambda(A+B)-\lambda(B)$.\\
If $\lambda(A)\leq\lambda(B)$ then  ${\rm{diam}}(B)\leq {\rm{diam}}(A)\leq \lambda(A+B)-\lambda(B)\leq \lambda(A+B)-\lambda(A)$ and we are done.\\ 
If $\lambda(A)>\lambda(B)$ then we have $\lambda(A+B)< \lambda(A)+\lambda(B)+\min(\lambda(A),\lambda(B))$ and the first part of this proof gives the result.
\end{itemize}
\end{proof}

We now state some lemma that we shall need to prove the last part of Theorem \ref{intro_3k-4}. For integer sets, a discrete analogue of this lemma was also used in the proof of this part. Our proof is very similar to the proof of the discrete version which can be found in \cite{Gr}.

\begin{lem}\label{lem_mes}
If $x\not\in A+B$ and $x\geq 0$ then 
\begin{equation*}
\lambda([0,x]\cap A)+\lambda([0,x]\cap B)\leq x.
\end{equation*}
If $x\not\in A+B$ and $x\leq D_A+D_B$ then 
\begin{equation*}
\lambda([x-D_B,D_A]\cap A)+\lambda([x-D_A,D_B]\cap B)\leq D_A+D_B-x.
\end{equation*}
\end{lem}
\begin{proof}[Proof of the lemma]
\begin{itemize}
\item If $x\not\in A+B$, then for all $b\in[0,x]$, we have either $b\not\in B$ or $x-b\not\in A$ thus $[0,x]\subset ([0,x]\setminus B)\cup ([0,x]\setminus (x-A))$. This yields \\$x\leq x-\lambda([0,x]\cap B)+x-\lambda([0,x]\cap A)$ and the first inequality.
\item We write $A'=D_A-A$, $B'=D_B-B$ and $x'=D_A+D_B-x$. If $x\not\in A+B$, then $x'\not\in A'+B'$ and an application of the first inequality yields the second one.
\end{itemize}
\end{proof}
\begin{proof}[Proof of the last part of Theorem \ref{intro_3k-4}]
We now turn to the end of the proof of Theorem \ref{intro_3k-4} and prove that under one of the two hypothesis of this theorem there exists an interval $I$ of length at least $\lambda(A)+\lambda(B)$ included in $A+B$. \\
We assume without loss of generality that $D_A\geq D_B$ and $\lambda(A+B)< \lambda(A)+2\lambda(B)$ (which is implied by $\lambda(A+B)< \lambda(A)+\lambda(B)+\min(\lambda(A),\lambda(B))$). The part of the theorem already proven together with any of the two hypothesis imply that $D_A\leq \lambda(A+B)-\lambda(B)<\lambda(A)+\lambda(B)$. We write $\Delta=\lambda(A)+\lambda(B)-D_A$. By hypothesis $\Delta>0$.\\
Reasoning modulo $D_A$ as Ruzsa does in  \cite{Ruz}, we write 
$$\lambda(A+B)=\mu_A(A+B)+\mu_A\left(\left\{x\in[0,D_A] \, :\, x,x+D_A\in A+B\right\}\right),$$
where $\mu_A$ denotes the Haar mesure modulo $D_A$.
Since $0,D_A \in A$, we have 
$$B\subset  \left\{x\in[0,D_B] \, :\, x,x+D_A\in A+B\right\}.$$
 Therefore
 \begin{equation}\label{mes_som}
 \lambda(A+B)\geq\mu_A(A+B)+\mu_A\left(\left\{x\in[0,D_A]\cap B^c \, :\, x,x+D_A\in A+B\right\}\right)+\mu_A(B).
 \end{equation}
For any positive real number $x$, we define 
$$g_A(x)=\lambda(A\cap[0,x]), \quad g_B(x)=\lambda(B\cap[0,x]), \quad $$
$$ g(x)=g_A(x)+g_B(x) \, \mbox{ and } \, h(x)=g_A(x+D_A-D_B)+g_B(x).$$
Lemma \ref{lem_mes} can be rephrased as follows :
\begin{equation}\label{majx_petit}
(x\not\in A+B, \quad x\geq 0) \Rightarrow g(x)\leq x,
\end{equation}
\begin{equation}\label{majx_grand}
 (y+D_A\not\in A+B, \quad 0\leq y\leq D_B) \Rightarrow h(y)\geq y+D_A-D_B+\Delta.  
\end{equation}
We first notice that $g$ and $h$ are non decreasing continuous positive functions. They are also $2$-Lipschitz functions.
We can divide the interval $[0,D_A]$ into $3$ areas: 
\begin{itemize}
\item $Z_1$ will be the set of $x\in[0,D_A]$ such that $g(x)\leq x$;
\item $Z_2$ will be the set of $x\in[0,D_A]$ such that $g(x)>x$ and $h(x)< x+D_A-D_B+\Delta$;
\item $Z_3$ will be the set of $x\in[0,D_A]$ such that $h(x)\geq x+D_A-D_B+\Delta$.
\end{itemize}

\bigskip
In the following picture, we draw two functions $g_A$ and $g_B$, the corresponding functions $g$ and $h$ and the corresponding areas $Z_1$, $Z_2$ and $Z_3$. The main part of the proof will consist in showing that with our hypothesis this drawing covers the possible configurations of the curves. More precisely, we shall prove that $[0,D_A]$ may be partitionned in three consecutive intervals $I_1$, $I_2$ and $I_3$ such that $Z_1\subset I_1\subset Z_1\cup Z_2$, $Z_3\subset I_3\subset Z_3\cup Z_2$ and $I_2\subset Z_2$.\\
First we explain why the large black interval $I_2\cup I_3\cup(D_A+I_1)\cup(D_A+I_2)$ traced below the curves will be contained in $A+B$ (in our drawing the sets $Z_i$ are the actual intervals $I_i$).

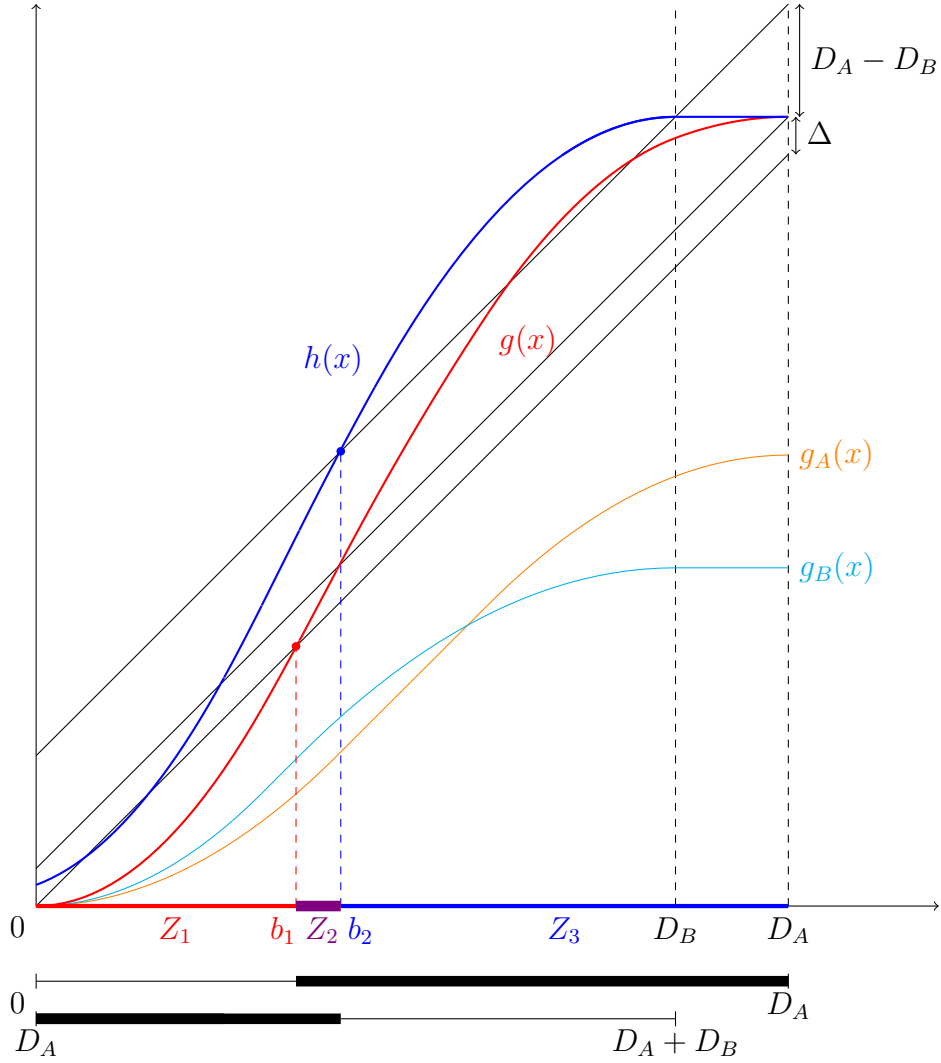
\bigskip

\begin{tikzpicture}[xscale=10,yscale=10]
\draw [<->] (0,1.2) -- (0,0) -- (1.2,0);
\node [below left] at (0,0) {$0$};
\draw [] (0,0) -- (1,1);
\draw [] (0,0.425+0.025+0.5+0.1-1) -- (1,0.425+0.025+0.5+0.1);
\draw [] (0,0.425+0.025+0.5+0.1-1+0.15) -- (1,0.425+0.025+0.5+0.1+0.15);
\draw [dashed] (1,0) -- (1,1.2);
\node [below] at (1,0) {$D_A$};
\draw [dashed] (0.85,0) -- (0.85,1.2);
\node [below] at (0.85,0) {$D_B$};

\draw [<->] (1.01,1+0.425+0.025+0.5+0.1-1) -- (1.01,1) node[above right]{$\Delta$};
\draw [<->] (1.015,1+0.425+0.025+0.5+0.1-1) -- (1.015,1+0.425+0.025+0.5+0.1-1+0.15);
\node[right] at (1.015,1+0.425+0.025+0.5+0.1-1+0.075) {$D_A-D_B$};

\draw[orange, domain=0:0.4] plot (\x,{\x*\x/0.8});
\draw[orange,domain=0.4:0.6] plot (\x,{\x-0.2});
\draw[orange,domain=0.6:1] plot (\x,{0.5+0.1-(1-\x)*(1-\x)/(2*(1-0.4-0.2))}) node[right] {$g_A(x)$};

\draw[cyan, domain=0:0.3] plot (\x,{\x*\x/0.6});
\draw[cyan,domain=0.3:0.35] plot (\x,{\x-0.15});
\draw[cyan,domain=0.35:0.85] plot (\x,{0.425+0.025-(0.85-\x)*(0.85-\x)/(2*(0.85-0.3-0.05))});
\draw[cyan,domain=0.85:1] plot (\x,{0.425+0.025}) node[right] {$g_B(x)$};

\draw[thick,red, domain=0:0.3] plot (\x,{\x*\x/0.6+\x*\x/0.8});
\draw[fill,red] (0.3455,0.3455) circle [radius=0.005] ;
\draw [dashed, red] (0.3455,0) -- (0.3455,0.3455);
\node [below left, red] at (0.36,0){$b_1$};
\draw[thick,red,domain=0.3:0.35] plot (\x,{\x-0.15+\x*\x/0.8});
\draw[thick,red,domain=0.35:0.4] plot (\x,{0.425+0.025-(0.85-\x)*(0.85-\x)/(2*(0.85-0.3-0.05))+\x*\x/0.8});
\draw[thick,red,domain=0.4:0.6] plot (\x,{0.425+0.025-(0.85-\x)*(0.85-\x)/(2*(0.85-0.3-0.05))+\x-0.2})node[below right] {$g(x)$};
\draw[thick,red,domain=0.6:0.85] plot (\x,{0.425+0.025-(0.85-\x)*(0.85-\x)/(2*(0.85-0.3-0.05))+0.5+0.1-(1-\x)*(1-\x)/(2*(1-0.4-0.2))});
\draw[thick,red,domain=0.85:1] plot (\x,{0.425+0.025+0.5+0.1-(1-\x)*(1-\x)/(2*(1-0.4-0.2))}) ;

\draw[thick,blue, domain=0:0.25] plot (\x,{\x*\x/0.6+(\x+0.15)*(\x+0.15)/0.8});
\draw[thick,blue, domain=0.25:0.3] plot (\x,{\x*\x/0.6+\x-0.2+0.15});
\draw[thick, blue,domain=0.3:0.35] plot (\x,{\x-0.15+\x-0.2+0.15});
\draw[fill,blue] (0.405,0.605) circle [radius=0.005] ;
\draw [dashed, blue] (0.405,0) -- (0.405,0.605);
\node [below right, blue] at (0.4,0){$b_2$};
\draw[thick, blue, domain=0.35:0.45] plot (\x,{0.425+0.025-(0.85-\x)*(0.85-\x)/(2*(0.85-0.3-0.05))+\x-0.2+0.15})node[above left] {$h(x)$};
\draw[thick, blue,domain=0.45:0.85] plot (\x,{0.425+0.025-(0.85-\x)*(0.85-\x)/(2*(0.85-0.3-0.05))+ 0.5+0.1-(1-(\x+0.15))*(1-(\x+0.15))/(2*(1-0.4-0.2))});
\draw[thick,blue,domain=0.7:0.85] plot (\x,{0.425+0.025-(0.85-\x)*(0.85-\x)/(2*(0.85-0.3-0.05))+0.5+0.1-(1-(\x+0.15))*(1-(\x+0.15))/(2*(1-0.4-0.2)) });
\draw[thick,blue,domain=0.85:1] plot (\x,{0.425+0.025+0.5+0.1});
   
 \draw[ultra thick, blue] (0.405,0) -- (1,0);
 \node [below, red] at (0.185,0){$Z_1$};
 \node [below, blue] at (0.7025,0){$Z_3$};
\node [below, violet] at (0.38,0){$Z_2$};
 \draw[ultra thick, red] (0.3455,0) -- (0,0);
  \draw[line width=4, violet] (0.3455,0) -- (0.405,0);
  
 \draw []  (0,-0.1) -- (1,-0.1);
 \draw []  (0,-0.11) -- (0,-0.09);
 \node [below left] at (0,-0.10){$0$};
  \draw []  (1,-0.11) -- (1,-0.09);
 \node [below ] at (1,-0.10){$D_A$};
 \draw[line width=4] (0.3455,-0.1) -- (1,-0.1);
 \draw[line width=4] (0,-0.15) -- (0.25,-0.15); 
  \draw[line width=4] (0.405,-0.15) -- (0,-0.15); 
\draw []  (0.405,-0.15) -- (0.85,-0.15);
\draw []  (0.85,-0.16) -- (0.85,-0.14);
\node [below ] at (0.85,-0.15){$D_A+D_B$};
 \draw []  (0,-0.16) -- (0,-0.14);
\node [below ] at (0,-0.15){$D_A$};

\end{tikzpicture}


\bigskip

Since $h(x)\leq g(x)+D_A-D_B$, by \eqref{majx_grand} and \eqref{majx_petit}, $Z_1$ is a subset of $A+B-D_A$, $Z_3$ is a subset of $A+B$ and $Z_2$ is a subset of $(A+B)\cap (A+B-D_A)$. Indeed
\begin{align*}
x\in Z_1&\Rightarrow h(x)\leq g(x)+D_A-D_B\leq x+D_A-D_B&\Rightarrow& \, x+D_A\in A+B&\\
x\in Z_3&\Rightarrow g(x)+D_A-D_B\geq h(x)\geq x+D_A-D_B +\Delta &\Rightarrow& \, g(x)>x \Rightarrow\, x\in A+B.
\end{align*}

Since $[0,D_A]=Z_1\cup Z_2\cup Z_3$ and $Z_i\subset A+B \bmod D_A$ for $i=1,2,3$, this leads to $\mu_A(A+B)=D_A$ and to $Z_2\subset \left\{x\in[0,D_A] \, :\, x,x+D_A\in A+B\right\}$.\\
With \eqref{mes_som} this yields
\begin{align}
\lambda(A+B)&\geq\mu_A(A+B)+\lambda(B)+\mu_A\left(\left\{x\in[0,D_A]\cap B^c \, :\, x,x+D_A\in A+B\right\}\right)\notag\\\label{minsum}
&\geq D_A+\lambda(B)+\lambda(Z_2\cap B^c)=\lambda(A)+2\lambda(B)+\lambda(Z_2\cap B^c)-\Delta.
\end{align}
We shall now prove that there exist $b_1,B_2\in [0,D_A]$ such that $0\leq b_1<b_1+\Delta\leq b_2\leq D_B\leq D_A$ and the intervals $I_1=[0,b_1]$, $I_2=(b_1,b_2)$ and $I_3=[b_2,D_A]$ satisfy $Z_1\subset I_1\subset Z_1\cup Z_2$, $Z_3\subset I_3\subset Z_3\cup Z_2$ and $I_2\subset Z_2$. This will imply that $I=(b_1,D_A+b_2)$ is an interval included in $A+B$ of size at least $D_A+\Delta=\lambda(A)+\lambda(B)$, hence the result.

\begin{itemize}
\item We first prove that $[D_B,D_A]\subset Z_3\cup Z_2\subset A+B$.\\ 
We have $g(D_A)=\lambda(A)+\lambda(B)=D_A+\Delta$ and for $x\in[D_B,D_A]$, we have 
$$g(x)=g(D_A)-\lambda(A\cap[x,D_A])\geq g(D_A)-(D_A-x)=x+\Delta>x$$
thus $x\in Z_3\cup Z_2$.
\item Now we prove that we can not switch from $Z_3$ to $Z_1$ when $x$ grows, meaning that for $x,y\in [0, D_B]$, $x\in Z_1$ and $y\in Z_3$ 
 imply $x<y$.
 We assume for contradiction that there exist $x,y\in[0, D_B]$ such that $x\geq y$, $ x\in Z_1$, $y\in Z_3$ and $(x,y)\in Z_2$.

  
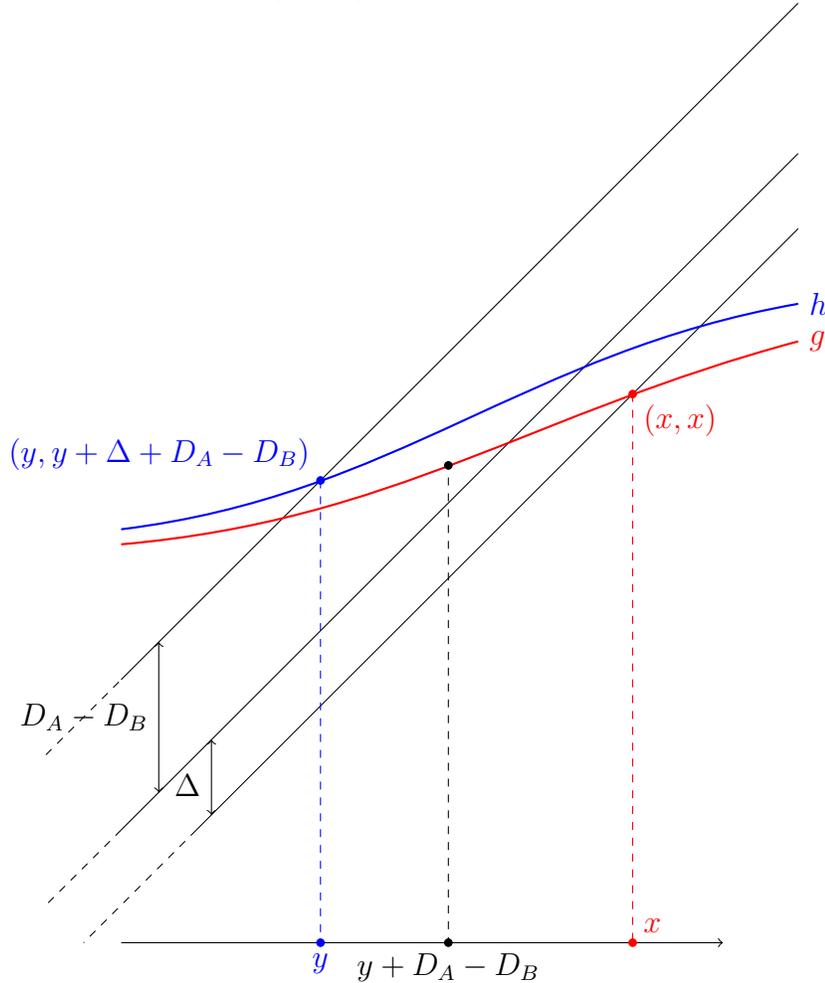
\begin{tikzpicture}[xscale=10,yscale=10]
\draw [->] (0.1,0.05) -- (0.9,0.05);
\draw [] (0.2,0.2) -- (1,1);
\draw [dashed] (0.2,0.2) -- (0.05,0.05);
\draw [dashed] (0.1,0.2) -- (0,0.1);
\draw [] (0.1,0.2) -- (1,1.1);
\draw [] (0.1,0.4) -- (1,1.3);
\draw [dashed] (0,0.3) -- (0.1,0.4);
\draw[thick, red] (0.1,0.58) to [out=5,in=195] (1,0.85) node[right]{$g$};
\draw[thick, blue] (0.1,0.6) to [out=7,in=190] (1,0.9) node[right]{$h$};
\draw[fill,red] (0.78,0.78) circle [radius=0.005] node [below right]{$(x,x)$};
\draw [dashed, red] (0.78,0.78) -- (0.78,0.05);
\draw[fill,red] (0.78,0.05) circle [radius=0.005] node [above right ]{$x$};
\draw[fill,blue] (0.365,0.665) circle [radius=0.005] node [above left]{$(y,y+\Delta+D_A-D_B)$};
\draw [dashed, blue] (0.365,0.665) -- (0.365,0.05);
\draw[fill,blue] (0.365,0.05) circle [radius=0.005] node [below ]{$y$};
\draw[fill] (0.535,0.05) circle [radius=0.005] node [below ]{$y+D_A-D_B$};
\draw [dashed,] (0.535,0.69) -- (0.535,0.05);
\draw[fill] (0.535,0.685) circle [radius=0.005];
\draw [<->] (0.22,0.22) -- (0.22,0.32);
\node[left] at (0.22,0.26){$\Delta$};
\draw [<->] (0.15,0.25) -- (0.15,0.45);
\node[left] at (0.15,0.35) {$D_A-D_B$};
 \end{tikzpicture}

We must have
 $g(x)\leq x$ and $h(y)\geq y+D_A-D_B+\Delta$, thus
$$g_A(y+D_A-D_B)-g_A(x)+g_B(y)-g_B(x)\geq y-x+(D_A-D_B) +\Delta.$$
Since $x\geq y$, $g_B$ is a non decreasing function and $\Delta>0$, this would lead to 
$$\lambda([x, y+D_A-D_B]\cap A)>\lambda([x, y+D_A-D_B])$$
in case $x\leq y+D_A-D_B$, a contradiction. Thus we must have $x> y+D_A-D_B$. 

Since $x> y+D_A-D_B$, we have
$$g_A(x)-g_A(y+D_A-D_B)+g_B(x)-g_B(y)\leq x-y-(D_A-D_B)-\Delta$$
which yields $\lambda(B\cap[y,x])\leq x-y-(D_A-D_B)-\Delta$ and
\begin{equation}\label{bc}
\lambda([y,x]\cap B^c)\geq\Delta+D_A-D_B.
\end{equation}
With \eqref{minsum}, we get $\lambda(A+B)\geq \lambda(A)+2\lambda(B)+D_A-D_B$
which contradicts the hypothesis. Since $0\in Z_1$ and $D_B\in Z_3$, we proved almost all the announced results on the intervals $I_i$. It remains to prove that $I_2$ has size at least $\Delta$. 
\item 
By continuity of $g$ and $h$, we have $h(b_2)=b_2+\Delta+D_A-D_B$ and $g(b_1)=b_1$. Since $g(b_2)=h(b_2)-\lambda(A\cap(b_2, b_2+D_A-D_B)\geq h(b_2)-(D_A-D_B)$ and since $g$ is a $2$-Lipschitz function, we have
$$2(b_2-b_1)\geq g(b_2)-g(b_1)\geq h(b_2)-(D_A-D_B)-g(b_1)= b_2+\Delta-b_1$$
thus $b_2-b_1\geq \Delta$ and $A+B$ contains an interval of size at least $\lambda(A)+\lambda(B)$.
\end{itemize}

\end{proof}

\section{Some observation on sets with large density.}
Our graphic interpretation for large sets of real numbers with small sumset gives rise to further comments. 

For this part, let $A$ and $B$ be some bounded closed subsets of real numbers such that $D_B\leq D_A$ and $\Delta:=\lambda(A)+\lambda(B)-D_A>0$. We define the functions $g$ and $h$ as in section 2.

As explained in section 2, we can divide $[0,D_A]$ into three sets $Z_1$, $Z_2$ and $Z_3$ corresponding to the three regions of $[0,D_A]\times[0,D_A]$ delimited by the lines $L_1$ and $L_2$ respectively defined by the equations $y=x$ and $y=x+\Delta$. $Z_1$ is the set of real numbers in $[0,D_A]$ such that the function $g$ is under the line $L_1$, $Z_3$ those for which $h$ is above $L_2$ and $Z_2$ what remains. $0$ is in $Z_1$, $D_A$ in $Z_3$ and to switch from $Z_1$ to $Z_3$ or reciprocally, one has to cross $Z_2$. 
One of the main ingredients in the proof of Theorem \ref{intro_3k-4} was to prove that there was at most one such crossing, therefore from $Z_1$ to $Z_3$ and that there were no crossing from $Z_3$ to $Z_1$. We shall call the crossings from $Z_1$ to $Z_3$ the "up crossings" and the crossings from $Z_3$ to $Z_1$ the "down crossings" (although the functions $g$ and $h$ remain nondecreasing functions). 

Let $m$ be the number of down crossings. 
We illustrate this by the following picture. For simplicity, we chose $D_A=D_B$ so that $g=h$ and $m=1$.

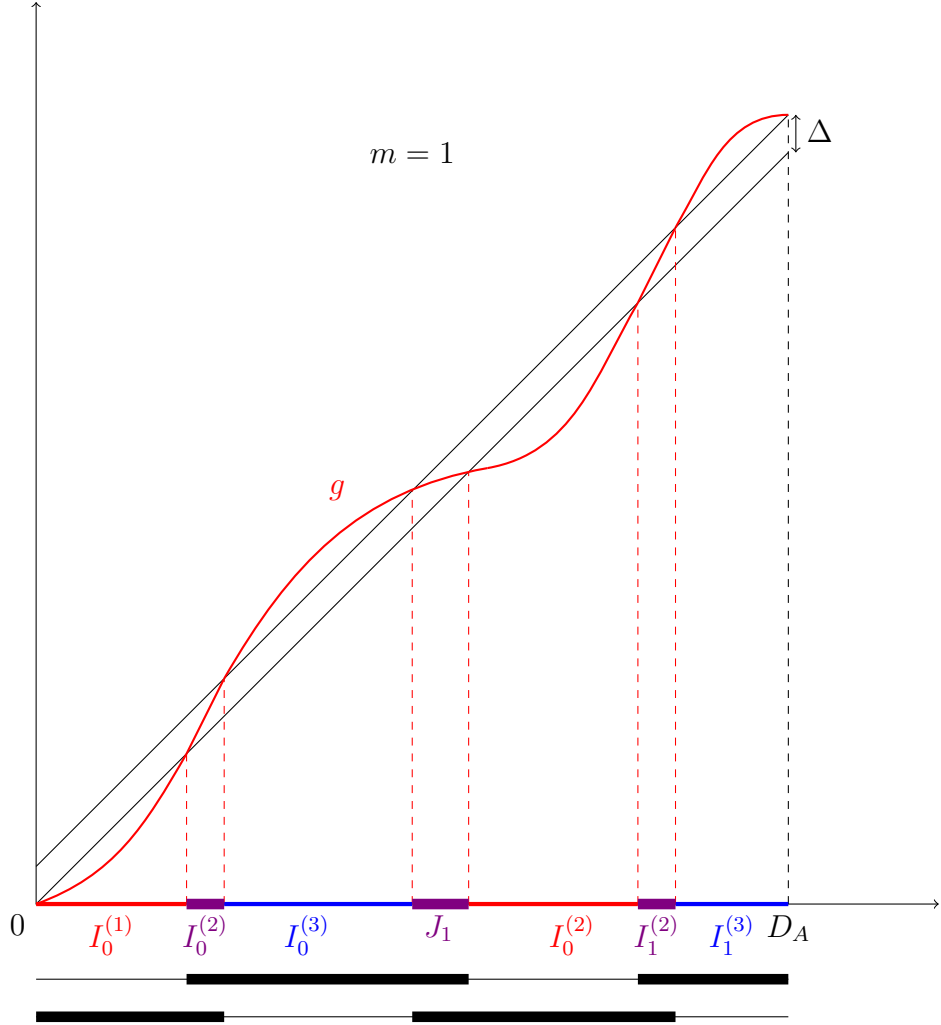
\begin{tikzpicture}[xscale=10,yscale=10]
\draw [<->] (0,1.2) -- (0,0) -- (1.2,0);
\node [below left] at (0,0) {$0$};
\draw [] (0,0) -- (1,1);
\draw [] (0,0.05) -- (1,1.05);
\draw [dashed] (1,0) -- (1,1.05);
\node [below] at (1,0) {$D_A$};
\node at (0.5,1) {$m=1$};
\draw [<->] (1.01,1+0.05) -- (1.01,1) node[above right]{$\Delta$};
\draw[thick, red] (0,0) to [out=20,in=240] (0.2,0.2) ;
\draw[thick,red,domain=0.2:0.25] plot (\x,{2*\x-0.2});
\draw[thick, red] (0.25,0.3) to [out=60,in=190] (0.6,0.58);
\node[red] at (0.4,0.55) {$g$};
\draw[thick, red] (0.6,0.58) to [out=10,in=240] (0.8,0.8) ;
\draw[thick,red,domain=0.8:0.85] plot (\x,{2*\x-0.8});
\draw[thick, red] (0.85,0.9) to [out=60,in=180] (1,1.05) ;
\draw [dashed, red] (0.25,0) -- (0.25,0.3);
\draw [dashed, red] (0.2,0) -- (0.2,0.2);
\draw [dashed, red] (0.575,0) -- (0.575,0.575);
\draw [dashed, red] (0.5,0) -- (0.5,0.55);
\draw [dashed, red] (0.8,0) -- (0.8,0.8);
\draw [dashed, red] (0.85,0) -- (0.85,0.9);
\draw[ultra thick, blue] (0.25,0) -- (0.5,0);
\draw[ultra thick, blue] (0.85,0) -- (1,0);
 \node [below, red] at (0.1,0){$I_0^{(1)}$};
 \node [below, red] at (0.7125,0){$I_0^{(2)}$};
  \node [below, violet] at (0.225,0){$I_0^{(2)}$};
    \node [below, violet] at (0.825,0){$I_1^{(2)}$};
 \node [below, violet] at (0.535,0){$J_1$};
 \node [below, blue] at (0.925,0){$I_1^{(3)}$};
\node [below, blue] at (0.36,0){$I_0^{(3)}$};
 \draw[ultra thick, red] (0.2,0) -- (0,0);
 \draw[ultra thick, red] (0.575,0) -- (0.8,0);
 \draw[line width=4, violet] (0.575,0) -- (0.5,0);
  \draw[line width=4, violet] (0.8,0) -- (0.85,0);
  \draw[line width=4, violet] (0.2,0) -- (0.25,0);
  
\draw []  (0,-0.1) -- (1,-0.1);
 \draw[line width=4] (0.2,-0.1) -- (0.575,-0.1);
 \draw[line width=4] (0.8,-0.1) -- (1,-0.1);
 \draw[line width=4] (0,-0.15) -- (0.25,-0.15); 
  \draw[line width=4] (0.5,-0.15) -- (0.85,-0.15); 
\draw []  (0,-0.15) -- (1,-0.15);

\end{tikzpicture}

\bigskip

We proved that for each down crossing, we gain a subset of $B^c$ of measure at least $\Delta+D_A-D_B$. It means that 
$$\lambda(A+B)\geq \lambda(B)+D_A+\lambda(B^c\cap Z_2)\geq  \lambda(B)+D_A+m(\Delta+D_A-D_B).$$
Furthermore, extending the previous remarks, we can write $[0,D_A]$ as a union of $4m+3$ consecutive intervals as follows:
$$[0,D_A]=I_0^{(1)}\cup I_0^{(2)}\cup I_0^{(3)}\cup\bigcup_{k=1}^m \left(J_k\cup I_k^{(1)}\cup I_k^{(2)}\cup I_k^{(3)}\right)$$
with $Z_1\subset\bigcup_{k=0}^m I_k^{(1)}\subset Z_1\cup Z_2$,  $Z_3\subset\bigcup_{k=0}^m I_k^{(3)}\subset Z_3\cup Z_2$ and $\bigcup_{k=0}^m I_k^{(2)}\cup \bigcup_{k=1}^m J_k\subset Z_2$.
The intervals $I_k^{(2)}$ correspond to up crossings whereas the intervals $J_k$ correspond down crossings. The length of up crossings is at least $\Delta$ whereas those of down crossings is at least $\Delta+D_A-D_B$.

Furthermore the set $A+B$ contains the folowing union of $2m+1$ intervals 
$$\bigcup_{k=1}^m\left( I_{k-1}^{(2)}\cup I_{k-1}^{(3)}\cup J_k\right) \cup\left(I_m^{(2)}\cup I_m^{(3)}\cup (D_A+ I_0^{(1)})\cup (D_A+I_0^{(2)})\right)\cup \left(D_A+\bigcup_{k=1}^m\left( J_k\cup I_{k}^{(1)}\cup I_{k}^{(2)}\right)\right)$$ 
Here each set in brackets is a single interval as a union of consecutive intervals.
 
ô\bigskip

These observations leads to Theorem \ref{intro_relaxed_3k-4}.
%
Note that even in the case $m=0$, this theorem gives a new information. In case $D_B<D_A$, Theorem \ref{intro_3k-4} needed $\lambda(A+B)<\lambda(A)+2\lambda(B)$ to conclude that $A+B$ contained an interval of size at least $\lambda(A)+\lambda(B)$ whereas Theorem \ref{intro_relaxed_3k-4} only needs $\lambda(A+B)<\lambda(A)+2\lambda(B)+D_A-D_B$ and $\lambda(A)+\lambda(B)>D_A$ to get the same conclusion.\\
\\
Some more elements on the structure of the sets $A$ and $B$ could be derived from the graphic interpretation we gave. For simplicity, we can assume $A=B$. In this case, we write $\lambda(A)=\frac12D_A+\delta$. The hypothesis of Theorem \ref{intro_relaxed_3k-4} becomes $\lambda(A+A)<D_A+\lambda(A)+2(m+1)\delta$. Since $\lambda(A+A)\leq 2D_A$ this hypothesis is fulfilled as soon as $\delta> \frac12\frac{3D_A}{2m+3}$.\\ In case $A=B$, 
 the set $[0,D_A]$ may be partitioned into the union of some disjoint intervals as follows
$$[0,D_A]=I_0^{(1)}\cup I_0^{(2)}\cup I_0^{(3)}\cup\bigcup_{k=1}^n \left(J_k\cup I_k^{(1)}\cup I_k^{(2)}\cup I_k^{(3)}\right),$$
$A$ has density $1/2$ of each interval $I_k^{(1)}$ and $I_k^{(3)}$, $\lambda(A\cap I_k^{(2)})=\frac12(\lambda(A\cap I_k^{(2)})+\Delta)$ and $\lambda(A\cap J_k)=\frac12(\lambda(A\cap J_k)-\Delta)$. Furthermore, there is a connexion in the structures of $A$ and $A+A$. This connexion is easier to explicite in the special case of extremal sets. This shall be the purpose of the next section.

\section{Small sumset and large densities: structure of the extremal sets.}
In \cite{Freiman2009IAP11}, Freiman exhibits a strong connexion in the description of $A$ and $A+A$ and reveals the structures of these sets of integers in case the size of $A+A$ is as small as it can be. In Theorem \ref{intro_equalarge} we give a similar results in the continuous setting. Our result also applies to sets $A$ and $B$ with $A\not=B$. As far as we know, no discrete analogue of this result can be found in the literature. 

\begin{proof}[Proof of Theorem \ref{intro_equalarge}]
We use the same notation as in the proof of Theorem \ref{intro_3k-4} and we assume again that $A$ and $B$ are closed bounded subsets of $\R$ such that $0=\inf {A}=\inf {B}$.
 We write $I_2=(b,D_B-c)$ with $g(b)=b$, $h(D_B-c)=D_B-c+\Delta$ and $\Delta:=\lambda(A)+\lambda(B)-D_A>0$.\\
We have on the one side
$$A+B=\left((A+B)\cap[0,b]\right)\cup (b,D_A+D_B-c)\cup \left((A+B)\cap[D_A+D_B-c,D_A+D_B]\right)$$
thus 
$$\lambda(A+B)=\lambda\left((A+B)\cap[0,b]\right)+D_A+D_B-(b+c) +\lambda\left((A+B)\cap[D_A+D_B-c,D_A+D_B]\right)$$
and on the other side 
$$\lambda(A+B)=D_B+\lambda(A)=D_B+\lambda\left(A\cap[0,b]\right)+\lambda\left(A\cap (b,D_A-c) \right)+\lambda\left(A\cap[D_A-c,D_A]\right).$$
Since $0,D_B\in B$, we have
$$A_1:=\left(A\cap[0,b]\right)\subset \left( (A+B)\cap[0,b]\right)\quad$$
and
$$D_A+D_B-A_2:= \left(D_B+\left(A\cap[D_A-c,D_A]\right)\right)\subset \left((A+B)\cap[D_A+D_B-c,D_A+D_B]\right)$$
thus
$$\left\{\begin{array}{l}
\lambda\left((A+B)\cap[0,b]\right)=\lambda\left(A_1\right), \\ 
\lambda\left((A+B)\cap[D_A+D_B-c,D_A+D_B]\right)=\lambda\left(A_2\right),\\ 
\lambda\left(A\cap (b,D_A-c) \right)=D_A-(b+c)\end{array}\right.$$
and up to a set of measure $0$, we have
$$\left\{\begin{array}{l}(A+B)\cap[0,b]=A\cap[0,b]=A_1, \\ (A+B)\cap[D_A+D_B-c,D_A+D_B]=D_B+\left(A\cap[D_A-c,D_A]\right)=D_B+D_A-A_2,\\ (b,D_A-c) \subset A.\end{array}\right.$$

Since $0, D_A\in A$, this in particular implies, up to a set of measure $0$, that
$$B\cap[0,b]\subset A_1 \quad \mbox{ and }\quad B\cap[D_B-c,D_B]\subset \left(D_B-A_2\right).$$

\bigskip

It remains now to prove the last part of the theorem concerning the density of $B$ near the border.
Let $u\in (0,b)$, we write $B_u=B\cap[0,u]$ and $A_u=A\cap[0,u]$. Since $0\in A\cap B$, these two sets are not empty. 

For $n\geq 1$ an integer, we have $B_u+A_{nu}\subset (A+B)\cap [0,(n+1)u]$ thus for $n\geq 1$ such that $(n+1)u<D_A-c$,
$$\lambda(B_u+A_{nu})\leq \lambda\left((A+B)\cap [0,(n+1)u]\right)=\lambda(A_{(n+1)u}).$$
Now we fix $u\in B\cap(0,b)$ such that $\lambda(B_u)>0$. For $n\geq 1$ we write
$$\frac{\lambda(A_{nu})}{\lambda(B_u)}=\frac{k_n(k_n-1)}{2}+k_n\delta_n=\frac{u_n(u_n-1)}2+\frac12\delta_n(1-\delta_n)$$
with $k_n\in\N^*$, $0\leq \delta_n<1$ and $u_n=k_n+\delta_n$. $(u_n)_{n\geq 0}$ is a non decreasing sequence and, since $\lambda(B_u)\leq \lambda(A_u)$,  we have $u_1\geq 2$.
 
We apply Ruzsa's lower bound \eqref{min_sum2}
$$\lambda(B_u+A_{nu})\geq \lambda(A_{nu})+\min(u_{n}\lambda(B_u), u).$$

Let $n_0$ be the maximum integer $n$ such that $(n+1)u\leq D_A+b_2-D_B$ and $u_n\leq u/\lambda(B_u)$.
For $n\leq n_0$, we have
$$\lambda(A_{(n+1)u})\geq \lambda(B_u+A_{nu})\geq\lambda(A_{nu})+u_n\lambda(B_u)$$
which implies $k_{n+1}\geq k_n+1$ and $u_{n+1}\geq u_n+1\geq n+u_1\geq n+2$.

Let $N$ be the maximal integer such that $\lambda(A_{Nu})\leq \frac12Nu$. 
Notice here that for $x\leq D_A-c$, we have $\lambda(A_x)\geq \frac12g(x)$ and for all $x> D_A-c-\Delta$, we have $g(x)> x$ thus $Nu<D_A-c$.

If $N\geq n_0+1$, then $\lambda(A_{Nu})\geq u(N-(n_0+1))+\lambda(A_{(n_0+1)u})$.
Now $u_{n_0+1}>u/\lambda(B_u)$ and $u_{n_0+1}-1\geq n_0+1$ thus 
 $\lambda(A_{Nu})\geq u(N-\frac12(n_0+1))$ which implies $N\leq n_0+1$ since $\lambda(A_{Nu})\leq \frac12Nu$.

$$\lambda(A_{Nu})\geq\frac{u_N(u_N-1)}{2}\lambda(B_u)\geq \frac12N(N+1)\lambda(B_u)$$
thus $\lambda(B_u)\leq u/(N+1)$.\\
Therefore, if $c=\sup\{x\, :\, \lambda(A_x)\leq \frac12x\}>0$, if $\lambda(B_u)>0$ and if we take $N=\max\{n\, : \, Nu\leq c\}$, we have 
$\lambda(B_u)\leq \frac{u}{N+1}\leq \frac{u^2}{c}$.\\
By symmetry and transitivity, we get the announced result.

\end{proof}

With the same arguments, we also can get the following result :
\begin{thm}
Let $A$ and $B$ be some bounded sets of real numbers such that $D_B\leq D_A$ and $\lambda(A+B)=D_A+\lambda(B)< \lambda(A)+2\lambda(B)$. Write $A'=A-\inf(A)$ and $B'=B-\inf(B)$.
Then there exists two positive real numbers $b$ and $c$ such that $b,c\leq D_B$, the interval $I=(b,D_B-c)$ has size at least $\lambda(A)+\lambda(B)-D_A=\Delta$ and, up to a set of measure $0$,
\begin{itemize}
\item $B_1=B'\cap[Õ0, b]$ is $A'$-stable, i.e. $B'\cap [0,b]=(A'+B')\cap[0,b]$;
\item $B_2=D_B-(B'\cap[D_B-c, D_B])$ is $A'$-stable;
\item $I=(b,D_B-c)\subset B'$;
\item $(b,D_A+D_B-c)\subset A'+B'$.
 \end{itemize}
 The sets $B'$ and $A'+B'$ may each be partitioned into three parts as follows
$$B'=B_1\cup I\cup (D_B-B_2), \quad 
A'+B'=B_1\cup(b,D_A+D_B-c)\cup(D_A+D_B-B_2)$$
 Furthermore $A'$ is a disjoint union of three sets $A'=A_1\cup A_I\cup D_A-A_2$ with 
 \begin{itemize}
\item $A_1\subset B_1$, $A_2\subset B_2$, $A_I\subset(b,D_A-c)$;
\item $\lambda(A_1\cap[0,x])\leq x^2/c_1$ for $x\leq c_1=\sup\{x\, :\, \lambda(B'\cap[0,x])\leq \frac12x\}$ 
\item $\lambda(A_2\cap[0,x])\leq x^2/c_2$ for $x\leq c_2=D_B-\sup\{x\, :\, \lambda(B'\cap(D_B-x,D_B))\leq \frac12x\}$
 \end{itemize}
 up to a set of measure $0$.
\end{thm}

These theorems are almost sharp as proved by the following examples.

Let $\delta$ be a positive real number and $a_1$, $a_2$ some strictly real numbers and $n$, $m$ some positive integers such that
\begin{equation}\label{cond1}
(n-1)a_1+(m-1)a_2=1-2\delta<1. 
\end{equation}
We define the following subset of $[0,1]$:
$$A:=\left(\bigcup_{k=0}^{n-1}A_k^{(1)}\right) \cup I \cup \left(\bigcup_{k=0}^{m-1} A_k^{(2)}\right)$$ 
with $A_k^{(1)}=\left[ka_1\left(1-\frac1n\right),ka_1\right]$, $A_k^{(2)}=\left[1-ka_2,1-ka_2\left(1-\frac1m\right)\right]$ and\\ $I=\left[(n-1)a_1,1-(m-1)a_2\right]$.\\
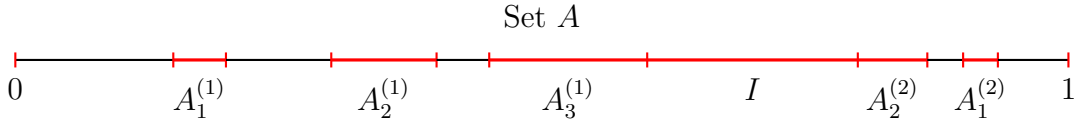
\bigskip

\begin{tikzpicture}[xscale=14]
\draw[-][thick] (0,0) -- (.15,0);
\draw [draw=red, thick] (0.15,-.1)  -- (0.15,0.1);
\node[align=center, below] at (.175,-.1) {$A_1^{(1)}$};
\draw[-][draw=red, very thick] (.15,0) -- (.2,0);
\draw [draw=red, thick] (0.2,-.1)  -- (0.2,0.1);
\draw[-][thick] (.2,0) -- (.3,0);
\draw [draw=red, thick] (0.3,-.1)  -- (0.3,0.1);
\draw[-][draw=red, very thick] (.3,0) -- (.4,0);
\node[align=center, below] at (.35,-.1) {$A_2^{(1)}$};
\draw [draw=red, thick] (0.4,-.1)  -- (0.4,0.1);
\draw[-][thick] (.4,0) -- (.45,0);
\draw [draw=red, thick] (0.45,-.1)  -- (0.45,0.1);
\draw[-][draw=red, very thick] (.45,0) -- (.6,0);
\node[align=center, below] at (.525,-.1) {$A_3^{(1)}$};
\draw[-][draw=red, very thick] (.6,0) -- (.8,0);
\draw[-][draw=red, very thick] (.8,0) -- (.866,0);
\node[align=center, below] at (.833,-.1) {$A_2^{(2)}$};
\draw [draw=red, thick] (0.866,-.1)  -- (0.866,0.1);
\draw[-][thick] (.866,0) -- (.9,0);
\draw [draw=red, thick] (0.9,-.1)  -- (0.9,0.1);
\draw[-][draw=red, very thick] (.9,0) -- (.933,0);
\node[align=center, below] at (.9165,-.1) {$A_1^{(2)}$};
\draw [draw=red, thick] (0.933,-.1)  -- (0.933,0.1);
\draw[-][thick] (.933,0) -- (1,0);
\draw [draw=red, thick] (0,-.1) node[below]{0} -- (0,0.1);
\draw [draw=red, thick] (0.6,-.1)  -- (0.6,0.1);
\draw [draw=red, thick] (0.8,-.1)  -- (0.8,0.1);
\node[align=center, below] at (.7,-.1) {$I$};
\draw [draw=red, thick] (1,-.1) node[below]{1} -- (1,0.1);
\node[align=center, above] at (.5,.3){Set $A$};
\end{tikzpicture}

\bigskip

The set $A$ is a closed subset of $[0,1]$ of diameter $1$ and of measure $\frac12+\delta$ by \eqref{cond1}.
We easily check that
$$A+A= \left(\bigcup_{k=0}^{n-1}A_k^{(1)}\right) \cup \left[(n-1)a_1,2-(m-1)a_2\right] \cup \left(\bigcup_{k=0}^{m-1}A_k^{(2)}\right)$$
thus $\lambda(A+A)=1+\lambda(A)=D_A+\lambda(A)<3\lambda(A)$.\\
For $x=Ka_1$ with $K\leq n-1$, we have
\begin{align*}
\lambda(A\cap[0,x])&=\frac{a_1}n\sum_{k=1}^{K}k=\frac{Ka_1(K+1)}{2n}=x\frac{K+1}{2n}
\end{align*}
thus $c_1:=\sup\{x\, :\, \lambda(A\cap[0,x])\leq \frac12x\}=(n-1)a_1$ and 
\begin{align*}
\lambda(A\cap[0,x])&=x\frac{Ka_1+a_1}{2na_1}=\frac{x(x+a_1)}{2(c_1+a_1)}\geq \frac{x^2}{2c_1}.
\end{align*}
This example shows that up to a factor $1/2$, the upper bound given in Theorem \ref{equalarge} for $\lambda(B_u)$ for small $u$ is optimal. 

One can not expect a result concerning an upper bound for $\lambda(A_u)$ similar to the result we got for $\lambda(B_u)$ in Theorem \ref{equalarge} when $A$ and $B$ are different sets. 
Actually, for positive real numbers $a$, $b$ and $\varepsilon$ and a positive integer $n$ such that $a/2>b>\varepsilon$, $1-na\geq b$ and $n\varepsilon\leq a-b$ we consider the following closed subsets of $[0,1]$ of diameter $1$:  
$$A=\{0\}\cup\bigcup_{k=1}^{n} \left[ak-b-(k-1)\varepsilon, ak\right]\cup[na,1],\quad
B=\bigcup_{k=0}^{n} \left[(a-\varepsilon)k, ak\right]\cup[na,1].$$

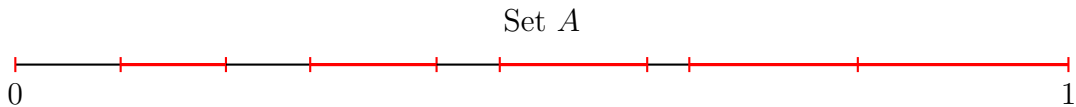
\begin{tikzpicture}[xscale=14]
\draw[-][thick] (0,0) -- (.1,0);
\draw [draw=red, thick] (0.1,-.1)  -- (0.1,0.1);
\draw[-][draw=red, very thick] (.1,0) -- (.2,0);
\draw [draw=red, thick] (0.2,-.1)  -- (0.2,0.1);
\draw[-][thick] (.2,0) -- (.28,0);
\draw [draw=red, thick] (0.28,-.1)  -- (0.28,0.1);
\draw[-][draw=red, very thick] (.28,0) -- (.4,0);
\draw [draw=red, thick] (0.4,-.1)  -- (0.4,0.1);
\draw[-][thick] (.4,0) -- (.46,0);
\draw [draw=red, thick] (0.46,-.1)  -- (0.46,0.1);
\draw[-][draw=red, very thick] (.46,0) -- (.6,0);
\draw[-][thick] (.6,0) -- (.64,0);
\draw[-][draw=red, very thick] (.64,0) -- (.8,0);
\draw[-][draw=red, very thick]  (.8,0) -- (1,0);
\draw [draw=red, thick] (0,-.1) node[below]{0} -- (0,0.1);
\draw [draw=red, thick] (0.6,-.1)  -- (0.6,0.1);
\draw [draw=red, thick] (0.8,-.1)  -- (0.8,0.1);
\draw [draw=red, thick] (0.64,-.1)  -- (0.64,0.1);
\draw [draw=red, thick] (1,-.1) node[below]{1} -- (1,0.1);
\node[align=center, above] at (.5,.3){Set $A$};
\end{tikzpicture}

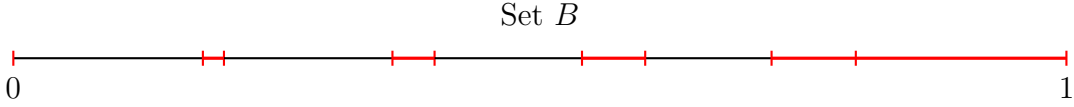
\begin{tikzpicture}[xscale=14]
\draw[-][thick] (0,0) -- (.18,0);
\draw [draw=red, thick] (0.18,-.1)  -- (0.18,0.1);
\draw[-][draw=red, very thick] (.18,0) -- (.2,0);
\draw [draw=red, thick] (0.2,-.1)  -- (0.2,0.1);
\draw[-][thick] (.2,0) -- (.36,0);
\draw [draw=red, thick] (0.36,-.1)  -- (0.36,0.1);
\draw[-][draw=red, very thick] (.36,0) -- (.4,0);
\draw [draw=red, thick] (0.4,-.1)  -- (0.4,0.1);
\draw[-][thick] (.4,0) -- (.54,0);
\draw [draw=red, thick] (0.54,-.1)  -- (0.54,0.1);
\draw[-][draw=red, very thick] (.54,0) -- (.6,0);
\draw[-][thick] (.6,0) -- (.72,0);
\draw[-][draw=red, very thick] (.72,0) -- (.8,0);
\draw[-][draw=red, very thick]  (.8,0) -- (1,0);
\draw [draw=red, thick] (0,-.1) node[below]{0} -- (0,0.1);
\draw [draw=red, thick] (0.6,-.1)  -- (0.6,0.1);
\draw [draw=red, thick] (0.8,-.1)  -- (0.8,0.1);
\draw [draw=red, thick] (0.72,-.1)  -- (0.72,0.1);
\draw [draw=red, thick] (1,-.1) node[below]{1} -- (1,0.1);
\node[align=center, above] at (.5,.3){Set $B$};
\end{tikzpicture}

\bigskip

We have 
$$A+B=\bigcup_{k=0}^{n} \left(ak-b-(k-1)\varepsilon, ak\right)\cup[na,2],$$
thus $\lambda(A+B)=1+\lambda(A)=D_B+\lambda(A)$. 
Furthermore, $\lambda(A)=1-na+nb+\frac12n(n-1)\varepsilon$ and $\lambda(B)=1-na+12n(n+1)\varepsilon$ thus $\lambda(A+B)<\lambda(A)+2\lambda(B)$ whenever $n(2a-(n+1)\varepsilon)>1$.
On the other hand for $x=aK\leq na$, we have
\begin{align*}
\lambda(A\cap[0,x])&=\sum_{k=1}^K(b+(k-1)\varepsilon)=Kb+\frac12K(K-1)\varepsilon\\
&=x\left(\frac{b}{a}+\frac{(K-1)\varepsilon}{2a}\right)\geq x\frac{b}{a}.
\end{align*}

\section{Small sets with small sumset: the extremal case. }
We now characterise the sets $A$ and $B$ such that equality holds in \eqref{min_sum2}, thus $\lambda(A+B)=\lambda(A)+(K+\delta)\lambda(B)$ with $K$ and $\delta$ defined in \eqref{defKdelta}. In \cite{Ruz}, Ruzsa gives an example of such sets $A$ and $B$. Theorem \ref{small_equa} states that his example is essentially the only kind of sets for which this equality holds.
Extremal sets will have the following shape (In this example, $K+3$ and $D_B=1$).

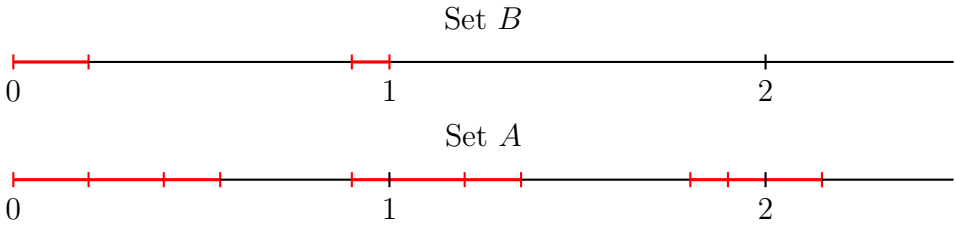
\begin{tikzpicture}[xscale=5]
\draw [draw=red, thick] (0.2,-.1)  -- (0.2,0.1);
\draw[-][draw=red, very thick] (0,0) -- (.2,0);
\draw[-][thick] (.2,0) -- (.9,0);
\draw [draw=red, thick] (0.9,-.1)  -- (0.9,0.1);
\draw[-][draw=red, very thick] (.9,0) -- (1,0);
\draw[-][thick] (1,0) -- (2.5,0);
\draw [draw=red, thick] (0,-.1) node[below]{0} -- (0,0.1);
\draw [draw=red, thick] (1,-.1) node[below]{1} -- (1,0.1);
\draw [thick] (2,-.1) node[below]{2} -- (2,0.1);
\node[align=center, above] at (1.25,.3){Set $B$};
\end{tikzpicture}

\begin{tikzpicture}[xscale=5]
\draw [draw=red, thick] (0.55,-.1)  -- (0.55,0.1);
\draw[-][draw=red, very thick] (0,0) -- (.55,0);
\draw[-][thick] (.55,0) -- (.9,0);
\draw [draw=red, thick] (0.9,-.1)  -- (0.9,0.1);
\draw [draw=red, thick] (1.9,-.1)  -- (1.9,0.1);
\draw[-][draw=red, very thick] (.9,0) -- (1.35,0);
\draw [draw=red, thick] (1.35,-.1)  -- (1.35,0.1);
\draw[-][thick] (1.35,0) -- (1.8,0);
\draw [draw=red, thick] (1.8,-.1)  -- (1.8,0.1);
\draw[-][draw=red, very thick] (1.8,0) -- (2.15,0);
\draw [draw=red, thick] (2.15,-.1)  -- (2.15,0.1);
\draw[-][thick] (2.15,0) -- (2.5,0);
\draw [draw=red, thick] (0,-.1) node[below]{0} -- (0,0.1);
\draw [thick] (1,-.1) node[below]{1} -- (1,0.1);
\draw [draw=red, thick] (0.2,-.1)--(0.2,0.1);
\draw [draw=red, thick] (0.4,-.1)--(0.4,0.1);
\draw [draw=red, thick] (1.2,-.1)--(1.2,0.1);
\draw [thick] (2,-.1) node[below]{2} -- (2,0.1);
\node[align=center, above] at (1.25,.3){Set $A$};
\end{tikzpicture}

\begin{proof}[Proof of Theorem \ref{small_equa}]
We assume without loss of generality that $A$ and $B$ are closed sets of $\R^+$ such that $0\in A,B$. Rescaling we also assume that $1\in B$ and ${\rm{diam}}(B)=1$. We write $\lambda$ for the Lebesgue measure on $\R$ and $\mu$ for the Haar measure on $\Tor=\R/\Z$. We write $S=A+B$ and for any positive integer $k$ and any subset $E$ of $\R^+$, we write 
$$\tilde{E}_k=\left\{x\in[0,1) \, : \, \#\{n\in\N \, : \, x+n\in E\}\geq k\right\}$$
 and  $K_E=\max\{k\in\N \, : \, \tilde{E}_k\not=\emptyset\}.$
 
 The proof will be divided into three steps.
 \begin{itemize}
\item The first step consists in determining the shape of $B$ and $\tilde{A}_k$ for positive integers $k$. To this aim, we shall follow Ruzsa's arguments in \cite{Ruz} and use Kneser's theorem in $\Tor$. \\
As noticed by Ruzsa, since $0,1\in B$, we have $\tilde{A}_{k-1}\subset \tilde{S}_k$ for $k\geq 2$ thus
\begin{equation}\label{minSk1_2}
\mu(\tilde{S}_k)\geq \mu(\tilde{A}_{k-1}) \quad (k\geq 2)
\end{equation}
and
$$\lambda(A)+1>\lambda(A+B)=\sum_{k=1}^{K_S}\mu(\tilde{S}_k)\geq \sum_{k=1}^{K_A}\mu(\tilde{A}_k)+\mu(\tilde{S}_1)=\lambda(A)+\mu(\tilde{S}_1)$$
thus  $\mu(\tilde{S}_1)<1$ and Raikov's theorem yields
\begin{equation}\label{minSk2}
\mu(\tilde{S}_k)\geq \mu(\tilde{A}_k)+\mu(B) \quad (k\leq K_A).
\end{equation}
Combining \eqref{minSk1_2} and \eqref{minSk2} leads to
$$\mu(\tilde{S}_k)\geq\frac{k-1}{K_A}\mu(\tilde{A}_{k-1})+ \frac{K_A+1-k}{K_A}(\mu(\tilde{A}_k)+\mu(B) )\quad (1\leq k\leq K_A+1)$$
and 
\begin{equation}\label{minSK}
\lambda(A+B)\geq \frac{K_A+1}{K_A}\lambda(A)+\frac{K_A+1}{2}\lambda(B),
\end{equation}
which is Ruzsa's lower bound.
Therefore, the hypothesis $\lambda(A+B)=\lambda(A)+(K+\delta)\lambda(B)=  \frac{K+1}{K}\lambda(A)+\frac{K+1}{2}\lambda(B)$ implies $K_A=K$,  and 
\begin{equation}\label{muSk}
\left\{
\begin{array}{ll}
\mu(\tilde{S}_k)= \mu(\tilde{A}_{k-1}) & (2\leq k\leq K+1)\\
\mu(\tilde{S}_k)= \mu(\tilde{A}_k)+\mu(B) & (1\leq k\leq K)\\
\mu(\tilde{S}_k)=0 &( k\geq K+2)
\end{array} 
\right.
\end{equation}
For $1\leq k\leq K$, we have $\tilde{A}_k+B\subset \tilde{S}_k$ thus the second line in \eqref{muSk} implies that we have equality in Macbeath's inequality, meaning $\mu(\tilde{A}_k+B)=\mu(\tilde{A}_k)+\mu(B)$ and by Kneser's theorem in $\Tor$ there exists $m_k\in\N^*$, there exist two closed intervals $I_k$ and $J_k$ of $\Tor$  such that  $m_k\tilde{A}_k\subset I_k$, $m_k B\subset J_k$ and $\mu(I_k)=\mu(\tilde{A}_k)$, $\mu(J_k)=\mu(B)$.\\
Now $m_k B\subset J_k$ with $\mu(J_k)=\mu(B)$ and $m_\ell B\subset J_\ell$ with $\mu(J_\ell)=\mu(B)$ implies $m_k=m_\ell$ and $J_k=J_\ell$. Let us write $J=J_k$ and $m=m_k$. We thus have for some $m\in\N^*$, up to sets of measure $0$ :\\
\begin{equation}\label{Sk}
\left\{
\begin{array}{ll}
m\tilde{A}_k= I_k, &  mB= J.\\
\tilde{S}_k= \tilde{A}_{k-1} & (2\leq k\leq K+1)\\
\tilde{S}_k= \tilde{A}_k+B & (1\leq k\leq K)\\
\tilde{S}_k=\emptyset &( k\geq K+2)
\end{array} 
\right.
\end{equation}
This implies $I_{k-1}=I_k+J$ for $2\leq k\leq K$ thus $I_k=I_K+(K-k)J$ for $1\leq k\leq K$. Since 
$$\lambda(A)=\sum_{k=1}^K\mu(\tilde{A}_k)=\sum_{k=1}^K\mu(m\tilde{A}_k)=\sum_{k=1}^K\mu(I_k),$$
we get, using $\mu(J)=\mu(B)$ and the definition of $K$ and $\delta$, that 
$$\mu(I_K)=\frac1K\lambda(A)-\frac{K-1}{2}\lambda(B)=\delta\lambda(B).$$
Now, we write $b=\lambda(B)$. We proved that we have, up to sets of measure $0$,
$$m\tilde{A}_k=I_k= I_K+(K-k)J,\, \mbox{ with }\,\mu(I_K)=\delta b\,\mbox{ and }\,\mu(J)=b.$$ 

Since $0\in B$, we can write $J=[-b_2,b_1]$ with $0\leq b_1,b_2$ and $b=b_1+b_2$. 
We also write $I_K=[\alpha,\beta]$ with $\beta-\alpha=\delta b$ and $0\in[\alpha-(K-1)b_2,\beta+(K-1)b_1]$ (since $0\in A$).
Hence we have 
$$I_k=[\alpha-(K-k)b_2, \beta+(K-k)b_1]$$ for $1\leq k\leq K$ and 
$\tilde{A}_k=\bigcup_{\ell=0}^mI_\ell^k \quad\mbox{ with }\quad$
$$ I_\ell^k=\frac{\ell}{m}+\frac1m I_k=\left[\frac1m(\ell+\alpha-(K-k)b_2),\frac1m(\ell+\beta+(K-k)b_1)\right],$$
whereas $B=\bigcup_{\ell=0}^mB_\ell\quad\mbox{ with }\quad $
$B_0=\left[0,\frac{b_1}{m}\right], \, B_m=\left[1-\frac{b_2}{m},1\right], \, $
$${\rm and} \, B_\ell=\left[\frac{\ell-b_2}{m},\frac{\ell+b_1}{m}\right] \, \mbox{ if }\, 1\leq\ell\leq m-1.$$

These sets are therefore of the following form.

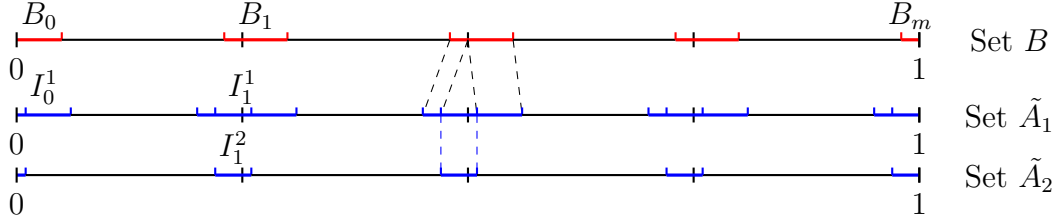
\begin{tikzpicture}[xscale=12]
\draw [thick] (0.25,0.9)  -- (0.25,1.1);
\draw [thick] (0.5,0.9)  -- (0.5,1.1);
\draw [thick] (0.75,0.9)  -- (0.75,1.1);
\draw [thick] (1,0.9)  -- (1,1.1);
\draw[-][draw=red, very thick] (0,1) -- (.05,1);
\node [above] at (0.025,1){$B_0$};
\draw [draw=red, thick] (0.05,1)  -- (0.05,1.1);
\draw[-][thick] (.05,1) -- (.23,1);
\draw [draw=red, thick] (0.23,1)  -- (0.23,1.1);
\draw[-][draw=red, very thick] (0.23,1) -- (.3,1);
\node [above] at (0.265,1){$B_1$};
\draw [draw=red, thick] (0.3,1)  -- (0.3,1.1);
\draw[-][thick] (.3,1) -- (.48,1);
\draw [draw=red, thick] (0.48,1)  -- (0.48,1.1);
\draw[-][draw=red, very thick] (0.48,1) -- (.55,1);
\draw [draw=red, thick] (0.55,1)  -- (0.55,1.1);
\draw[-][thick] (.55,1) -- (.73,1);
\draw [draw=red, thick] (0.73,1)  -- (0.73,1.1);
\draw[-][draw=red, very thick] (0.73,1) -- (.8,1);
\draw [draw=red, thick] (0.8,1)  -- (0.8,1.1);
\draw[-][thick] (.8,1) -- (.98,1);
\draw [draw=red, thick] (0.98,1)  -- (0.98,1.1);
\draw[-][draw=red, very thick] (0.98,1) -- (1,1);
\draw [thick] (0,0.9) node[below]{0} -- (0,1.1);
\draw [thick] (1,0.9) node[below]{1} -- (1,1.1);
\node [above] at (0.99,1){$B_m$};
\node[align=center] at (1.1,1){Set $B$};

\draw [dashed] (0.55,1) -- (0.56,0);
\draw [dashed] (0.5,1) -- (0.51,0);
\draw [dashed] (0.5,1) -- (0.47,0);
\draw [dashed] (0.48,1) -- (0.45,0);

\draw [thick] (0.25,-.1)  -- (0.25,0.1);
\draw [thick] (0.5,-.1)  -- (0.5,0.1);
\draw [thick] (0.75,-.1)  -- (0.75,0.1);
\draw [thick] (1,-.1)  -- (1,0.1);
\draw[-][draw=blue, very thick] (0,0) -- (.06,0);
\node [above] at (0.03,0){$I_0^1$};
\draw [draw=blue, thick] (0.01,0)  -- (0.01,.1);
\draw [draw=blue, thick] (0.06,0)  -- (0.06,.1);
\draw[-][thick] (.06,0) -- (.2,0);
\draw [draw=blue, thick] (0.22,0)  -- (0.22,.1);
\draw [draw=blue, thick] (0.2,0)  -- (0.2,.1);
\draw[-][draw=blue, very thick] (0.2,0) -- (.31,0);
\node [above] at (0.25,0){$I_1^1$};
\draw [draw=blue, thick] (0.26,0)  -- (0.26,.1);
\draw [draw=blue, thick] (0.31,0)  -- (0.31,.1);
\draw[-][thick] (.31,0) -- (.45,0);
\draw [draw=blue, thick] (0.45,0)  -- (0.45,.1);
\draw [draw=blue, thick] (0.47,0)  -- (0.47,.1);
\draw[-][draw=blue, very thick] (0.45,0) -- (.56,0);
\draw [draw=blue, thick] (0.51,0)  -- (0.51,.1);
\draw [draw=blue, thick] (0.56,0)  -- (0.56,.1);
\draw[-][thick] (.56,0) -- (.7,0);
\draw [draw=blue, thick] (0.7,0)  -- (0.7,.1);
\draw [draw=blue, thick] (0.72,0)  -- (0.72,.1);
\draw[-][draw=blue, very thick] (0.7,0) -- (.81,0);
\draw[-][thick] (.81,0) -- (.95,0);
\draw [draw=blue, thick] (0.76,0)  -- (0.76,.1);
\draw [draw=blue, thick] (0.81,0)  -- (0.81,.1);
\draw[-][draw=blue, very thick] (0.95,0) -- (1,0);
\draw [draw=blue, thick] (0.95,0)  -- (0.95,.1);
\draw [draw=blue, thick] (0.97,0)  -- (0.97,.1);
\draw [thick] (0,-.1) node[ below]{0} -- (0,0.1);
\draw [thick] (1,-.1) node[below]{1} -- (1,0.1);
\node[align=center] at (1.1,0){Set $\tilde{A_1}$};

\draw [thick] (0.25,-.7)  -- (0.25,-.9);
\draw [thick] (0.5,-.7)  -- (0.5,-.9);
\draw [thick] (0.75,-.7)  -- (0.75,-.9);
\draw [thick] (1,-.7)  -- (1,-.9);
\draw[-][draw=blue, very thick] (0,-.8) -- (.01,-.8);
\draw [draw=blue, thick] (0.01,-.8)  -- (0.01,-.7);
\draw[-][thick] (.01,-.8) -- (.22,-.8);
\draw [draw=blue, thick] (0.22,-.8)  -- (0.22,-.7);
\draw[-][draw=blue, very thick] (0.22,-.8) -- (.26,-.8);
\node [above] at (0.24,-.8){$I_1^2$};
\draw [draw=blue, thick] (0.26,-.8)  -- (0.26,-.7);
\draw[-][thick] (.26,-.8) -- (.47,-.8);
\draw [draw=blue, thick] (0.47,-.8)  -- (0.47,-.7);
\draw[-][draw=blue, very thick] (0.47,-.8) -- (.51,-.8);
\draw [draw=blue, thick] (0.51,-.8)  -- (0.51,-.7);
\draw[-][thick] (.51,-.8) -- (.72,-.8);
\draw [draw=blue, thick] (0.72,-.8)  -- (0.72,-.7);
\draw[-][draw=blue, very thick] (0.72,-.8) -- (.76,-.8);
\draw[-][thick] (.76,-.8) -- (.97,-.8);
\draw [draw=blue, thick] (0.76,-.8)  -- (0.76,-.7);
\draw[-][draw=blue, very thick] (0.97,-.8) -- (1,-.8);
\draw [draw=blue, thick] (0.97,-.8)  -- (0.97,-.7);
\draw [thick] (0,-.9) node[ below]{0} -- (0,-.7);
\draw [thick] (1,-.9) node[below]{1} -- (1,-.7);
\node[align=center] at (1.1,-.8){Set $\tilde{A_2}$};

\draw [dashed, blue] (0.51,0) -- (0.51,-0.8);
\draw [dashed, blue] (0.47,0) -- (0.47,-0.8);


\end{tikzpicture}

Note that we also have 
$\tilde{S}_k=\bigcup_{\ell=0}^mI_\ell^k \quad\mbox{ with }\quad$
$$ I_\ell^k=\frac{\ell}{m}+\frac1m I_k=\left[\frac1m(\ell+\alpha-(K+1-k)b_2),\frac1m(\ell+\beta+(K+1-k)b_1)\right],$$
for $0\leq k\leq K$, where the definition of $I_0$ is the obvious extension of those of $I_K$ for $k\geq 1$.
 
Since $0\in A$, we also can write
$I_1=[-a_2,a_1]$ with $0\leq a_1,a_2$ and $a_1+a_2=(K-1+\delta) b$.
We thus have 
$$A_\ell=A\cap\left[\frac{\ell-a_2}{m},\frac{\ell+a_1}{m}\right] \, \mbox{ if }\, \ell\geq 0$$ 
and
$$A+B= \bigcup_{\ell\geq 0}\bigcup_{k=0}^{m}(A_\ell+B_k)= \bigcup_{i\geq0}S_i$$
with $$S_i=S\cap\left[\frac{i-(a_2+b_2)}{m},\frac{i+(a_1+b_1)}{m} \right]=\bigcup_{k,\ell\geq 0\, :\, k+\ell=i}(A_\ell+B_k). $$

We now introduce 
$$A_\ell^k=A_\ell\cap \left[\frac\ell{m}+\frac{\alpha-(K-k)b_2}{m}, \frac\ell{m}+\frac{\beta+(K-k)b_1}{m}\right]$$ and
$$S_\ell^k=S_\ell\cap \left[\frac\ell{m}+\frac{\alpha-(K-k)b_2}{m}, \frac\ell{m}+\frac{\beta+(K-k)b_1}{m}\right].$$

\item As a second step, we shall prove that for any positive integer $\ell$, the set $A_\ell^K$ is either empty or equal to $I_{\ell}^K$.\\
For $r$ such that $0\leq r\leq m-1$, we write $L(r)$ for the set of integers $\ell\geq 0$ such that $A$ does not intersect
$$\left[\frac{\ell m+r}{m}+\max\left(\frac{\alpha}{m};\frac{\beta-b_1(1+\delta)/2}{m}\right), \frac{\ell m+r}{m}+\min\left(\frac{\beta}{m};\frac{\alpha+b_2(1+\delta)/2}{m}\right)\right].$$
Since $\left[\frac{r}{m}+\max\left(\frac{\alpha}{m};\frac{\beta-b_1(1+\delta)/2}{m}\right), \frac{r}{m}+\min\left(\frac{\beta}{m};\frac{\alpha+b_2(1+\delta)/2}{m}\right)\right]\subset I_r^K\subset\tilde{A_K}$, the set $L(r)$ contains at least $K$ elements.\\
Furthermore $A+B_0\subset S$, hence for $\ell\in L(r)$ we must have 
$$\left[\frac{\ell m+r}{m}+\frac{\beta}{m}, \frac{\ell m+r}{m}+\frac{\beta}{m}+\frac{b_1(1-\delta)/2}{m}\right]\subset S_{\ell m+r}^{K-1}.$$
Now $$\left(\frac{r}{m}+\frac{\beta}{m}, \frac{r}{m}+\frac{\beta}{m}+\frac{b_1(1-\delta)/2}{m}\right)\subset \tilde{S}_{K}\setminus \tilde{S}_{K+1},$$ thus for $\ell\not\in L(r)$
$$\left(\frac{\ell m+r}{m}+\frac{\beta}{m}, \frac{\ell m+r}{m}+\frac{\beta}{m}+\frac{b_1(1-\delta)/2}{m}\right)\cap S=\emptyset.$$
Since $A+B_0\subset S$, it implies that for $\ell\not\in L(r)$
\begin{equation}\label{IK1}
\left(\frac{\ell m+r}{m}+\frac{\beta}{m}-\frac{b_1}m, \frac{\ell m+r}{m}+\frac{\beta}{m}+\frac{b_1(1-\delta)/2}{m}\right)\cap A=\emptyset.
\end{equation}
Similarly, since $A+B_m\subset S$, for $\ell\in L(r)$ we must have 
$$\left[\frac{(\ell+1) m+r}{m}+\frac{\alpha}{m}-\frac{b_2(1-\delta)/2}{m}, \frac{(\ell+1) m+r}{m}+\frac{\alpha}{m}\right]\subset S_{(\ell+1) m+r}^{K-1}$$
and
$$\left(\frac{(\ell+1) m+r}{m}+\frac{\alpha}{m}-\frac{b_2(1-\delta)/2}{m}, \frac{(\ell+1) m+r}{m}+\frac{\alpha}{m}\right)\cap S=\emptyset$$
for $\ell\not\in L(r)$ and since $A+B_m\subset S$, it implies that for $\ell\not\in L(r)$
\begin{equation}\label{IK2}
\left(\frac{\ell m+r}{m}+\frac{\alpha}{m}-\frac{b_2(1-\delta)/2}{m}, \frac{\ell m+r}{m}+\frac{\alpha}{m}+\frac{b_2}{m}\right)\cap A=\emptyset
\end{equation}
Now, $\beta-\alpha=\delta b<b=b_1+b_2$ so $$\frac{\ell m+r}{m}+\frac{\alpha}{m}+\frac{b_2}{m}>\frac{\ell m+r}{m}+\frac{\beta}{m}-\frac{b_1}m$$
thus \eqref{IK1} and \eqref{IK2} imply  for $\ell\not\in L(r)$ that
$$\left(\frac{\ell m+r}{m}+\frac{\alpha}{m}-\frac{b_2(1-\delta)/2}{m},\frac{\ell m+r}{m}+\frac{\beta}{m}+\frac{b_1(1-\delta)/2}{m} \right)\cap A=\emptyset.$$
In particular 
$$A_{\ell m+r}^{K}=\begin{cases}I_{\ell m+r}^{K} & \mbox{ if } \ell\in L(r)\\\emptyset &\mbox{ otherwise.}\end{cases}$$

\item In this third part of the proof, we prove that $L(r)$ is a set of $K$ consecutive integers and we extend the previous arguments to determine the shape of $A$ and $S$. We conclude the proof by showing that $m=1$.

Since $0,1\in B$, for any $r$ such that $0\leq r\leq m-1$ and $\ell\geq 0$, we have 
$$A_{\ell m+r}^{K}\subset S_{\ell m+r}^{K} \quad \mbox{ and }\quad 1+A_{\ell m+r}^{K}\subset S_{(\ell+1) m+r}^{K},$$
and since 
$\tilde{S}_{K+2}=\emptyset$, we have that for any $r\in [0,m-1]$, $\#\{L(r)\cup(1+L(r))\}\leq K+1$, thus $L(r)$ is the set of $K$ consecutive integers and there exists $\ell_0(r)\geq 0$ such that $L(r)=\{\ell_0(r),\ell_0(r)+1,\cdots, \ell_0(r)+K-1\}$.
For $\ell\in L(r)$, 
$$\left[\frac{\ell m+r}{m}+\frac{\alpha}{m}, \frac{\ell m+r}{m}+\frac{\beta}{m}\right] \subset A$$ thus
$$\left[\frac{\ell m+r}{m}+\frac{\beta}{m}, \frac{\ell m+r}{m}+\frac{\beta+b_1}{m}\right] \subset S_{\ell m+r}^{K-1}$$ thus for $\ell\not\in L(r)$, we have
$$\left(\frac{\ell m+r}{m}+\frac{\beta}{m}, \frac{\ell m+r}{m}+\frac{\beta+b_1}{m}\right) \cap S= \emptyset$$ 
which implies that for $\ell\not\in \{\ell_0(r),\ell_0(r)+1,\cdots, \ell_0(r)+K-2\}$, we have 
$$\left(\frac{\ell m+r}{m}+\frac{\beta}{m}, \frac{\ell m+r}{m}+\frac{\beta+b_1}{m}\right) \cap A= \emptyset$$ 
because $A,A+1\subset S$.
Similarly, using $A+B_m\subset S$, we get
for $\ell\not\in \{\ell_0(r),\ell_0(r)+1,\cdots, \ell_0(r)+K-2\}$, we have 
$$\left(\frac{(\ell+1) m+r}{m}+\frac{\alpha-b_1}{m}, \frac{(\ell+1) m+r}{m}+\frac{\alpha}{m}\right) \cap A= \emptyset$$ 
and 
$$\left(\frac{\ell_0(r) m+r}{m}+\frac{\alpha-b_1}{m}, \frac{\ell_0(r) m+r}{m}+\frac{\alpha}{m}\right) \cap A= \emptyset.$$ 
Therefore we have
$$A_{\ell m+r}^{K-1}=\begin{cases}
\left[\frac{\ell m+r}{m}+\frac{\alpha-b_2}{m}, \frac{\ell m+r}{m}+\frac{\beta+b_1}{m}\right] & \mbox{ if } \ell_0(r)+1\leq\ell\leq\ell_0(r)+K-2\\
\left[\frac{\ell m+r}{m}+\frac{\alpha-b_2}{m}, \frac{\ell m+r}{m}+\frac{\beta}{m}\right] &\mbox{ if } \ell=\ell_0(r)+K-1\\
\left[\frac{\ell m+r}{m}+\frac{\alpha}{m}, \frac{\ell m+r}{m}+\frac{\beta+b_1}{m}\right] & \mbox{ if } \ell=\ell_0(r)\\
 \emptyset &\mbox{ otherwise.}\end{cases}$$
Iterating these arguments, we get that
$$A=\bigcup_{r=0}^{m-1}\bigcup_{j=0}^{K-1}\left[\frac{(j+\ell_0(r)) m+r}{m}+\frac{\alpha-jb_2}{m}, \frac{(j+\ell_0(r)) m+r}{m}+\frac{\beta+(K-1-j)b_1}{m}\right]$$
and
$$S=\bigcup_{r=0}^{m-1}\bigcup_{j=0}^{K}\left[\frac{(j+\ell_0(r)) m+r}{m}+\frac{\alpha-jb_2}{m}, \frac{(j+\ell_0(r)) m+r}{m}+\frac{\beta+(K-j)b_1}{m}\right].$$
Since $S=A+B$, for this to be true, we need $m=1$. Actually, if it were not the case, then $A+B_1\subset S$ would imply that $\ell_0(r)$ is independant of $r$ and that
 $\left[\frac{\ell_0 m+1}{m}+\frac{\alpha-b_2}{m}, \frac{\ell_0 m+1}{m}+\frac{\beta+Kb_1}{m}\right]\subset S$, which is not the case.
 Therefore $m=1$ and
 $$A=\bigcup_{j=0}^{K-1}\left[j+\ell_0+\alpha-jb_2, j+\ell_0+\beta+(K-1-j)b_1\right].$$
 Since $0=\inf A$, we must have $\ell_0=0$ and $\alpha=0$ and
 $$A=\bigcup_{j=0}^{K-1}\left[j-jb_2, j+\delta b+(K-1-j)b_1\right].$$
 With such an $A$, we easily check that
 $$A+B=\bigcup_{j=0}^{K}\left[j-jb_2, j+\delta b+(K-j)b_1\right]$$
 and $\lambda(A+B)=(K+1)\delta b+b\frac{K(K+1)}{2}=\lambda(A)+(K+\delta)B$.

\end{itemize}

\end{proof}

\section{Acknowledgement}
The author warmly thanks Professor Ruzsa for very useful conversations. She is also indebted to Professor Serra who gave a talk at the conference Additive Combinatorics in Bordeaux which helped her to get a perspective on some works in the discrete setting related to hers. He also helped her to improve the presentation of this paper. She is also grateful to Pablo Candela for bringing to her knowledge Grynkiewicz's book and for his careful reading of this paper.

\bibliographystyle{alpha}
\bibliography{biblio_Freiman}

\begin{thebibliography}{Ruz91}

\bibitem[BG10]{Bardaji_Grynkiewicz}
Itziar Bardaji and David~J. Grynkiewicz.
\newblock Long arithmetic progressions in small sumsets.
\newblock {\em Integers}, 10:A28, 335--350, 2010.

\bibitem[Chr]{Cr}
Michael Christ.
\newblock Near equality in the riesz-sobolev inequality.
\newblock {\em arXiv}.

\bibitem[Fre59]{Freiman1959}
G.~A. Fre{\u\i}man.
\newblock The addition of finite sets. {I}.
\newblock {\em Izv. Vys\v s. U\v cebn. Zaved. Matematika}, 1959(6
  (13)):202--213, 1959.

\bibitem[Fre62]{Freiman1962}
G.~A. Fre{\u\i}man.
\newblock Inverse problems of additive number theory. {VI}. {O}n the addition
  of finite sets. {III}.
\newblock {\em Izv. Vys\v S. U\v cebn. Zaved. Matematika}, 1962(3
  (28)):151--157, 1962.

\bibitem[Fre09]{Freiman2009IAP11}
Gregory~A. Freiman.
\newblock Inverse additive number theory. {XI}. {L}ong arithmetic progressions
  in sets with small sumsets.
\newblock {\em Acta Arith.}, 137(4):325--331, 2009.

\bibitem[Gry13]{Gr}
David~J. Grynkiewicz.
\newblock {\em Structural additive theory}, volume~30 of {\em Developments in
  Mathematics}.
\newblock Springer, Cham, 2013.

\bibitem[LS95]{Lev_Smeliansky}
Vsevolod~F. Lev and Pavel~Y. Smeliansky.
\newblock On addition of two distinct sets of integers.
\newblock {\em Acta Arith.}, 70(1):85--91, 1995.

\bibitem[Nat96]{Nathanson}
Melvyn~B. Nathanson.
\newblock {\em Additive number theory}, volume 165 of {\em Graduate Texts in
  Mathematics}.
\newblock Springer-Verlag, New York, 1996.
\newblock Inverse problems and the geometry of sumsets.

\bibitem[Rai39]{Raikov}
D.~Raikov.
\newblock On the addition of point-sets in the sense of {S}chnirelmann.
\newblock {\em Rec. Math. [Mat. Sbornik] N.S.}, 5(47):425--440, 1939.

\bibitem[Ruz91]{Ruz}
Imre~Z. Ruzsa.
\newblock Diameter of sets and measure of sumsets.
\newblock {\em Monatsh. Math.}, 112(4):323--328, 1991.

\bibitem[Sta96]{stanchescu}
Yonutz Stanchescu.
\newblock On addition of two distinct sets of integers.
\newblock {\em Acta Arith.}, 75(2):191--194, 1996.

\bibitem[TV06]{Tao_Vu}
Terence Tao and Van Vu.
\newblock {\em Additive combinatorics}, volume 105 of {\em Cambridge Studies in
  Advanced Mathematics}.
\newblock Cambridge University Press, Cambridge, 2006.

\end{thebibliography}


%
%
%
%
%
%
%
%
%
%

\end{document}